\documentclass{amsart}
\usepackage{amssymb}
\usepackage{esint}
\usepackage{enumerate}
\usepackage{hyperref}

\newtheorem{theorem}{Theorem}[section]
\newtheorem{lemma}[theorem]{Lemma}
\newtheorem{corollary}[theorem]{Corollary}
\newtheorem{proposition}[theorem]{Proposition}

\theoremstyle{definition}

\theoremstyle{remark}
\newtheorem{remark}[theorem]{Remark}

\numberwithin{equation}{section}

\makeatletter
\@namedef{subjclassname@2020}{%
  \textup{2020} Mathematics Subject Classification}
\makeatother

\begin{document}

\title[Second Derivative of Navier-Stokes]{Second derivatives estimate of suitable solutions to the 3D Navier-Stokes equations}

\author[A. Vasseur]{Alexis Vasseur}
\address{Department of Mathematics, The University of Texas at Austin, 2515 Speedway Stop C1200 Austin, TX 78712, USA}
\email{vasseur@math.utexas.edu}

\author[J. Yang]{Jincheng Yang}
\address{Department of Mathematics, The University of Texas at Austin, 2515 Speedway Stop C1200 Austin, TX 78712, USA}
\email{jcyang@math.utexas.edu}


\thanks{\textit{Acknowledgement}. A.~Vasseur was partially supported by the NSF grant: DMS 1614918.}

\date{\today}

\begin{abstract}
We study the second spatial derivatives of suitable weak solutions to the incompressible Navier-Stokes equations in dimension three. We show that it is locally $L ^{\frac43, q}$ for any $q > \frac43$, which improves from the current result $L ^{\frac43, \infty}$. Similar improvements in Lorentz space are also obtained for higher derivatives of the vorticity for smooth solutions. We use a blow-up technique to obtain nonlinear bounds compatible with the scaling. The local study works on the vorticity equation and uses De Giorgi iteration. In this local study, we can obtain any regularity of the vorticity without any a priori knowledge of the pressure. The local-to-global step uses a recently constructed maximal function for transport equations.

\bigskip

\end{abstract}

\maketitle
\tableofcontents

\newcommand{\grad}{\nabla}
\renewcommand{\div}{\operatorname{div}}
\newcommand{\R}{\mathbb{R}}
\newcommand{\Rd}{\R ^d}
\newcommand{\vp}{\varphi}
\renewcommand*{\d}{\mathop{\kern0pt\mathrm{d}}\!{}}
\newcommand{\e}{\varepsilon}
\newcommand{\rad}{1}
\newcommand{\br}{B _{\rad}}
\newcommand{\be}{B _{\e}}
\newcommand{\Qe}{Q _\e}
\newcommand{\bex}{\be (x)}
\newcommand{\RpRt}{\R _+ \times \R ^3}
\renewcommand{\L}[2]{L ^#1 (#2)}
\newcommand{\tx}{(t, x)}

\newcommand{\uuhalf}{\frac{|u|^2}{2}}
\newcommand{\vvhalf}{\frac{|v|^2}{2}}
\newcommand{\BB}{\mathbf{B}}
\newcommand{\BL}{\mathbf{L}}
\newcommand{\BW}{\mathbf{W}}
\newcommand{\BC}{\mathbf{C}}
\newcommand{\Cu}{\BC _u}
\newcommand{\Cv}{\BC _v}
\newcommand{\RR}{\mathbf{R}}
\newcommand{\loc}{\mathrm{loc}}
\newcommand{\Lip}{\mathrm{Lip}}
\newcommand{\diam}{\mathrm{diam}}

\newcommand{\pt}{\partial _t}
\newcommand{\La}{\Delta}
\newcommand{\curl}{\operatorname{curl}}
\renewcommand{\div}{\operatorname{div}}
\newcommand{\supp}{\operatorname{supp}}
\newcommand{\Id}{\operatorname{Id}}
\newcommand{\inv}{^{-1}}
\newcommand{\tensor}{\otimes}
\newcommand{\cross}{\times}
\newcommand{\vps}{\vp ^\sharp}
\newcommand{\Pcurl}{\mathbb{P} _{\curl}}
\newcommand{\Pgrad}{\mathbb{P} _{\grad}}
\newcommand{\pq}[1]{L ^{p _#1} _t L ^{q _#1} _x}
\newcommand{\Lp}[1]{L ^{p _#1} _t}
\newcommand{\Lq}[1]{L ^{q _#1} _x}
\newcommand{\LLLH}{L ^\infty _t L ^2 _x \cap L ^2 _t \dot H ^1 _x}
\newcommand{\LH}{L ^2 _t \dot H ^1 _x}
\newcommand{\ind}[1]{\mathbf1 _{#1}}
\newcommand{\inds}[1]{\ind{\{#1\}}}

\newcommand{\psis}[1]{\varrho ^{\sharp #1} _0}
\newcommand{\psisn}[1]{\varrho ^{\sharp #1} _n}

\newcommand{\vr}{\varrho}
\newcommand{\vs}{\varsigma}
\newcommand{\vrn}{\vr _n}
\newcommand{\vsn}{\vs _n}

\newcommand{\BMO}{\textit{BMO}}

\newcommand{\vpe}{\vp _\e}
\newcommand{\ue}{u _\e}
\newcommand{\Xe}{X _\e}
\newcommand{\mm}{\mathcal{M}}
\newcommand{\mmu}{\mm \left(|\grad u|\right)}
\newcommand{\mmut}{\mm \left(|\grad u (t)|\right)}
\newcommand{\mmus}{\mm \left(|\grad u (s)|\right)}
\newcommand{\betxo}{\bar \e _{\toxo}}
\newcommand{\mmq}{\mm _\mathcal{Q}}

\newcommand{\intR}[1]{\int _{\R ^#1}}
\newcommand{\cci}{C _c ^\infty}

\section{Introduction}
We study the three dimensional incompressible Navier-Stokes equations,
\begin{align}
    \label{eqn:ns}
    \pt u + u \cdot \grad u + \grad P = \La u, \qquad \div u = 0.
\end{align}
Here $u: (0, T) \times \R ^3 \to \R ^3$ and $P: (0, T) \times \R ^3 \to \R$ represent the velocity field and the pressure field of a fluid in $\R ^3$, within a finite or infinite timespan of length $T$. Initial condition
\begin{align*}
    u (0, \cdot) = u _0 \in L ^2 (\R ^3)
\end{align*}
is given by a divergence-free velocity profile $u _0$ of finite energy.

Leray (\cite{Leray1934}) and Hopf (\cite{Hopf1951}) proved the existence of weak solutions for all time. They constructed solutions $u \in C _w (0, \infty; L ^2 (\R ^3)) \cap L ^2 (0, \infty; \dot H ^1 (\R ^3))$ corresponding to each aforementioned initial value, and satisfying \eqref{eqn:ns} in the sense of distribution. A weak solution is called a Leray-Hopf solution if it satisfies energy inequality
\begin{align*}
    \frac12 \| u (t) \| _{L ^2 (\R ^3)} ^2 + 
    \| \grad u \| _{L ^2 ((0, t) \times \R ^3)} ^2 \le 
    \frac12 \| u _0 \| _{L ^2 (\R ^3)} ^2 
\end{align*}
for every $t > 0$. Since then, much work has been developed in regard to the uniqueness and regularity of weak solutions. Nonuniqueness of weak solutions was proven very recently by Buckmaster and Vicol (\cite{Buckmaster2019}) using convex integration scheme. However, the question of the uniqueness of Leray-Hopf solutions still remains open. The uniqueness is related with the regularity of solutions by the Lady\v zenskaya-Prodi-Serrin criteria (\cite{Ladyzenskaya1957,Prodi1959,Serrin1962,Serrin1963,Fabes1972}): if the velocity belongs to any space interpolating $L ^2 _t L ^\infty _x$ and $L ^\infty _t L ^3 _x$ then it is actually smooth, hence unique. The endpoint case $L ^\infty _t L ^3 _x$ comes much later by Iskauriaza, Ser\"egin and Shverak \cite{Iskauriaza2003}. These spaces require $\frac16$ higher spatial integrability than the energy space provides, which is $\mathcal E = L ^\infty _t L ^2 _x \cap L ^2 _t \dot H ^1 _x$.

At the level of energy space, Scheffer (\cite{Scheffer1976,Scheffer1977,Scheffer1978,Scheffer1980}) began to study the partial regularity for a class of Leray-Hopf solutions, called suitable weak solutions. These solutions exist globally and satisfy the following local energy inequality,
\begin{align*}
    \pt \frac{|u|^2}{2} + \div\left(u\left(
        \frac{|u|^2}{2} + P
    \right)\right) + |\grad u| ^2 \le \La \frac{|u|^2}{2}.
\end{align*}
Scheffer showed the singular set, at which the solution is unbounded nearby, has time-space Hausdorff dimension at most $\frac53$. This result was later improved by Caffarelli, Kohn and Nirenberg in \cite{Caffarelli1982} (see also \cite{Lin1998,Vasseur2007}), where they showed the 1-dimensional Hausdorff measure of the singular set is zero. We will investigate the regularity of suitable weak solutions. In the periodic setting, Constantin (\cite{Constantin1990}) constructed suitable weak solutions whose second derivatives have space-time integrability $L ^{\frac43 - \e}$ for any $\e > 0$, provided the initial vorticities are bounded measures. This was improved by Lions (\cite{Lions1996}) to a slightly better space $L ^{\frac43, \infty}$, a Lorentz space which corresponds to weak $L ^\frac43$ space. These estimates are extended to higher derivatives of smooth solutions by one of the authors and Choi (\cite{Vasseur2010,Choi2014}) using blow-up arguments: $L ^{p, \infty} _{\loc}$ space-time boundedness for $(-\La) ^\frac\alpha2 \grad ^n u$, where $p = \frac4{n + \alpha + 1}$, $n \ge 1$, $0 \le \alpha < 2$. They also constructed suitable weak solutions satisfying these bounds for $n + \alpha < 3$.

The aim of this paper is to improve these regularity results in Lorentz space. The main result is the following. Note that the estimate does not rely on the size of the pressure.

\begin{theorem}
\label{thm:main}
Suppose we have a smooth solution $u$ to the Navier-Stokes equations in $(0, T) \times \R ^3$ for some $0 < T \le \infty$ with smooth divergence free initial data $u _0 \in L ^2$. Then for any integer $n \ge 0$, for any real number $q > 1$, the vorticity $\omega = \curl u$ satisfies
\begin{align}
\label{eqn:main-thm}
    \left\|
        |\grad ^n \omega| ^\frac{4}{n+2} \inds{|\grad ^n \omega| ^\frac{4}{n+2} > C _n t^{-2}}
    \right\| _{L ^{1, q}((0, T) \times \R ^3)} \le C _{q, n} \| u _0 \| _{L ^2} ^2
\end{align}
for some constant $C _n$ depending on $n$ and $C _{q, n}$ depending only on $q$ and $n$, uniform in $T$. The above estimate \eqref{eqn:main-thm} also holds for suitable weak solutions with only $L ^2$ divergence free initial data in the case $n = 1$.
\end{theorem}

This theorem gives the following improvement on the second derivatives.
\begin{corollary}
\label{cor:L43}
Let $u$ be a suitable weak solution in $(0, \infty) \times \R ^3$ with initial data $u _0 \in L ^2$. Then for any $q > \frac43$, $K \subset \subset (0, \infty) \times \R ^3$, there exists a constant $C _{q, K}$ depending on $q$ and $K$ such that the following holds,
\begin{align*}
    \left\|
        \grad ^2 u
    \right\| _{L ^{\frac43, q} (K)} \le C _{q, K} \left(
        \| u _0 \| _{L ^2} ^\frac32 + 1
    \right).
\end{align*}
\end{corollary}

Let us explain the main ideas of the proof. Similar as previous work on higher derivatives, the proof is also based on blow-up techniques. In particular, we blow up the equation along a trajectory, using the scaling symmetry and the Galilean invariance of the Navier-Stokes equations. That is, if we fix an initial time $t _0$ and move the frame of reference along some $X (t)$, and zoom in into $\e$ scale, then it is easy to verify that $\tilde u (s, y)$, $\tilde P (s, y)$ defined by
\begin{align}
    \label{eqn:tildeu}
    \frac1\e \tilde u \left(\frac{t - t _0}{\e ^2}, \frac{x - X (t)}{\e} \right) &:= u (t, x) - \dot X(t) \\
    \notag
    \frac1{\e ^2} \tilde P \left(\frac{t - t _0}{\e ^2}, \frac{x - X (t)}{\e} \right) &:= P(t, x) + x \cdot \ddot X(t)
\end{align}
also satisfy the Navier-Stokes equation
\begin{align*}
    \partial _s \tilde u + \tilde u \cdot \grad \tilde u + \grad \tilde P = \La \tilde u, \qquad \div \tilde u = 0.
\end{align*}
We develop the following local theorem for $\tilde u$ and $\tilde P$. Note that it needs nothing from the pressure. Denote $B _r \subset \R ^3$ to be a ball centered at the origin with radius $r$, and $Q _r = (-r ^2, 0) \times B _r \subset \R ^4$ to be a space-time cylinder. 

\begin{theorem}[Local Theorem]
\label{thm:local}
There exists a universal constant $\eta _1 > 0$, such that for any suitable weak solution $u$ to the Navier-Stokes equations in $(-4, 0) \cross \R ^3$ satisfying
\begin{align}
    \label{eqn:mean-zero-velocity}
    &\int _{B _1} u (t, x) \phi (x) \d x = 0 \qquad \text{\upshape a.e. } t \in (-4, 0), \\
    \label{eqn:compensated-integrability}
    &\| \grad u \| _{\pq1 (Q _2)} + \| \omega \| _{\pq2 (Q _2)} \le \eta _1,
\end{align}
where $\phi \in \cci (B _1)$ is a non-negative function with $\int \phi = 1$, $\omega = \curl u$ is the vorticity, $\frac43 \le p _1 \le \infty$, $1 \le p _2 \le \infty$, $1 \le q _1, q _2 < 3$ satisfying
\begin{align*}
    \frac1{p _1} + \frac1{p _2} < 1, \qquad 
    \frac1{q _1} + \frac1{q _2} \le \frac76,
\end{align*}
then for any integer $n \ge 0$, we have
\begin{align*}
    \| \grad ^n \omega \| _{L ^\infty (Q _{8 ^{-n-2}})} \le C _n
\end{align*}
for some constant $C _n$ depending only on $n$.
\end{theorem}

Let us illustrate the ideas of how to go from this local theorem towards the main result. We want to choose a ``pivot quantity'', blow up near a point, and use this quantity to control $\grad ^n \omega$. When we patch the local results together, we will obtain a nonlinear bound with the same scaling as the pivot quantity, so we want the pivot quantity to have the best possible scaling. The ideal pivot quantities would be $\int |\grad u| ^2 \d x \d t$ and $\int |\grad ^2 P| \d x \d t$. $\int |u| ^\frac{10}3 \d x \d t$ has a worse scaling and should not be used. However, we still need to control the flux in the local theorem, so we want to take out the mean velocity and control $u$ by $\grad u$ using Poincar\'e's inequality.

In order to take out the mean velocity, we choose $X (t)$ to be the trajectory of the mollified flow so that \eqref{eqn:mean-zero-velocity} can be realized. Notice that a cylinder $Q _r$ in the local $(s, y)$ coordinate will be transformed into a ``skewed cylinder'' growing along $X (t)$ in the global $(t, x)$ coordinate. One of the authors recently constructed a maximal function $\mmq$ associated with these cylinders (\cite{Yang2020}), which serves as a bridge between the local theorem and the global result, and is one of the main reasons for the improvement in this paper. The idea is, if locally the vorticity gradient can be controlled in $L ^\infty$ by the integral of something in the skewed cylinder, and the integral in a skewed cylinder can be controlled by the maximal function $\mmq$, then vorticity gradient is pointwisely bounded by the maximal function. 

If one uses $\int |\grad u| ^2 \d x \d t$ and $\int |\grad ^2 P| \d x \d t$ as the pivot quantity, then unfortunately the best possible outcome would just be an $L ^{1, \infty}$ bound, as obtained in \cite{Yang2020}. The reason is, the maximal function is bounded on $L ^p$ for $p > 1$, but for $p = 1$ it is only bounded from $L ^1$ to $L ^{1, \infty}$. Unfortunately $|\grad u| ^2$ and $|\grad ^2 P|$ are both $L ^1$ quantities, so $\mmq \left(|\grad u| ^2 + |\grad ^2 P|\right)$ is only $L ^{1, \infty}$. We need two things to improve from $L ^{1, \infty}$: replace $\int |\grad u| ^2$ by $\int |\grad u| ^p$, and drop the pressure $\grad ^2 P$. 

Suppose we could use $(\int |\grad u| ^p \d x \d t) ^\frac2p$ as the pivot quantity for some $p < 2$, then we can majorize it by $\mmq \left(|\grad u| ^p\right) ^\frac2p \in L ^1$, since $\frac2p>1$ and $\mmq$ is bounded in $L ^\frac2p$. However, this poses significant difficulties in the local theorem. The nonlinear term $u \cdot \grad u$ is quadratic, and if we only have a subquadratic integrability to begin with, we cannot treat this quadratic transport term as a source term because it is not integrable. Observe that what we lack of is the temporal integrability rather than the spatial one: if $p$ is slightly smaller than two, than $u \cdot \grad u$ is still $L ^{\frac32-}$ in space, but $L ^{1-}$ in time. To overcome this difficulty, we write $u \cdot \grad u$ as $\omega \cross u$ up to a gradient term, and put $L ^{2-} _t L^{6-} _x$ on $u$ and $L ^{2+} _t L ^{2-} _x$ on $\omega$. We compensate the lower integrability term by pairing with a higher integrability term to make $\omega \cross u$ integrable. $L ^{2+} _t L ^{2-} _x$ of $\omega$ can be interpolated between $L ^{2-} _t L^{2-} _x$ and $L ^\infty _t L ^1 _x$, while the latter is controlled by $L ^2 _{t, x}$ of $\grad u$. Since $L ^{2+} _t L ^{2-} _x$ is closer to $L ^{2-} _t L^{2-} _x$ than to $L ^\infty _t L ^1 _x$, the pivot quantity that we use is actually $\delta ^{-\nu} \| \grad u \| ^2 _{L ^p} + \delta \| \grad u \| ^2 _{L ^2}$ for $\nu$ close to 0. By using more subquadratic integrability and a tiny bit of the quadratic one, we can complete the task by interpolation. That is why we obtain $L ^{1, q}$ in the end: it interpolates $L ^1$ bound from $\| \grad u\| _{L ^p}$ and $L ^{1, \infty}$ bound from $\| \grad u \| _{L ^2}$. Unfortunately we still miss the endpoint $L ^1$.

The second task is more subtle and technical. Without any information on the pressure, we don't have any control on the nonlocal effect. However, the role of the pressure is not important at the vorticity level: if we take the curl of the Navier-Stokes equation, the pressure will disappear and we are left with the vorticity equation involving only local quantities:
\begin{align}
    \label{eqn:vorticity}
    \pt \omega + u \cdot \grad \omega - \omega \cdot \grad u = \La \omega.
\end{align}
Inspired by Chamorro, Lemari\'e-Rieusset and Mayoufi (\cite{Chamorro2018}), we introduce a new velocity variable $v = -\curl \vps \La \inv \vp \omega$ using only local information of vorticity ($\vp$ and $\vps$ are spatial cut-off functions), and this helps us to prove the local theorem. This is another main reason for the improvement in this paper. Consequently, the bounds we obtain in the end is on the vorticity $\omega$ rather than on the velocity $u$.

This paper is organized as follows. In the preliminary Section 2 we introduce the analysis tools to the reader. We show how to rigorously derive the main results from the local theorem in Section 3, and then deal with technicalities of the local theorem in the later sections. The proof of the local theorem consists of three parts. Section 4 introduces the new variables $v$, and shows the smallness of $v$ in the energy space. Then we use De Giorgi iteration argument in Section 5 to prove boundedness of $v$. Finally, we inductively bound $\omega$ and all its higher derivatives in Section 6.

\section{Preliminary}

In this section, we introduce a few tools that we are going to use in the paper, including the maximal function, Lorentz space, and Helmholtz decomposition. 

\subsection{Maximal Function associated with Skewed Cylinders} 
This is recently developed for incompressible flows in \cite{Yang2020}. We quote useful results here without proof.

Suppose $u \in L ^p (0, T; \dot W ^{1, p} (\R ^3; \R ^3))$ is a vector field in $\R ^3$. Fix $\phi \in C _c ^\infty (B _1)$ to be a nonnegative function with $\int \phi = 1$ through out the paper. For $\e > 0$ define $\phi _\e (x) = \e ^{-3} \phi(-x/\e)$, and let $u _\e (t, \cdot) = u (t, \cdot) * \phi _\e$ be the mollified velocity. For a fixed $(t, x)$ we let $X (s)$ solve the following initial value problem,
\begin{align*}
    \begin{cases}
        \dot X (s) = u _\e (s, X (s)), \\
        X (t) = x.
    \end{cases}
\end{align*}
The skewed parabolic cylinder $Q _\e (t, x)$ is then defined to be
\begin{align}
\label{eqn:def-Qe}
    Q _\e (t, x) := \left\lbrace
        (t + \e ^2 s, X (t) + \e y) : -9 \le s \le 0, y \in B _3
    \right\rbrace.
\end{align}
We use $\mm$ to denote the spatial Hardy-Littlewood maximal function, which is defined by
\begin{align*}
    \mm (f)(t, x) = \sup _{r > 0} \fint _{B _r(x)} |f(t, y)| \d y.
\end{align*}
Then we construct the space-time maximal function adapted to the flow.

\begin{theorem}[$\mathcal Q$-Maximal Function]
\label{thm:maximal-function}
There exists a universal constant $\eta _0$ such that the following is true. We say $Q _\e (t, x)$ is admissible if $Q _\e (t, x) \subset (0, T) \times \R ^3$ and 
\begin{align}
    \label{eqn:admissibility}
    \e ^2 \fint _{Q _\e (t, x)} \mm (|\grad u|) \d x \d t \le \eta _0.
\end{align}
Define the maximal function 
\begin{align*}
    \mmq (f) (t, x) := \sup _{\e > 0} \left\lbrace 
        \fint _{Q _\e (t, x)} |f(s, y)| \d s \d y: Q _\e (t, x) \text{ is admissible}
    \right\rbrace.
\end{align*}
If $u$ is divergence free and $\mm (|\grad u|) \in L ^q$ for some $1 \le q \le \infty$, then $\mmq$ is bounded from $L ^1 ((0, T) \times \R ^3)$ to $L ^{1, \infty} ((0, T) \times \R ^3)$ and from $L ^p ((0, T) \times \R ^3)$ to itself for any $p > 1$ with norm depending on $p$.
\end{theorem}

An important consequence of the weak type $(1,1)$ bound of the Hardy-Littlewood maximal function is the Lebesgue differentiation theorem in $\R ^n$. Similarly, we can use the $\mathcal Q$-maximal function to prove the $\mathcal Q$-Lebesgue differentiation theorem.

\begin{theorem}[$\mathcal Q$-Lebesgue Differentiation Theorem]
\label{thm:lebesgue}
Let $f \in L ^1 _{\loc} ((0, T) \times \R ^3)$. Then for almost every $(t, x) \in (0, T) \times \R ^3$,
\begin{align*}
    \lim _{\e \to 0} \fint _{Q _\e (t, x)} |f (s, y) - f (t, x)| \d s \d y = 0.
\end{align*}
In this case we say $(t, x)$ is a $\mathcal Q$-Lebesgue point of $f$.
\end{theorem}

\subsection{Lorentz Space}

Let $(X, \mu)$ be a measure space. Recall that for a measurable function $f$, its decreasing rearrangement is defined as
\begin{align*}
    f ^* (\lambda) := \inf \left\lbrace
        \alpha > 0: \mu (\{|f| > \alpha\}) < \lambda
    \right\rbrace, \qquad \lambda \ge 0.
\end{align*}
For $0 < p < \infty$, $0 < q \le \infty$, Lorentz space $L ^{p,q} (X)$ is defined as the set of functions $f$ for which
\begin{align*}
    \| f \| _{L ^{p,q} (X)} := \| t ^\frac1p f ^* \| _{L ^q (\frac{\d t}{t})} < \infty.
\end{align*}
Now we introduct the interpolation lemma for Lorentz spaces.

\begin{lemma}[Interpolation of Lorentz Spaces]
\label{lem:lorentz}
Let $\nu > 0$ be a fixed positive number. Assume $f _0 \in L ^{p _0, q _0}$, $f _1 \in L ^{p _1, q _1}$, where $0 < p _0, p _1 < \infty, 0 < q _0, q _1 \le \infty$. If $f$ is a measurable function satisfying
\begin{align*}
    2|f| \le \delta f _0 + \delta ^{-\nu} f _1 \qquad \forall \delta > 0
\end{align*}
then $f \in L ^{p, q}$, where
\begin{align*}
    \frac1p = \frac \nu{1+\nu} \frac1{p _0} + \frac 1{1+\nu} \frac1{p _1}, \qquad \frac1q = \frac \nu{1+\nu} \frac1{q _0} + \frac 1{1+\nu} \frac1{q _1}.
\end{align*}
\end{lemma}

\begin{proof}
Upon on decreasing rearrangement, we may assume $f, f _0, f _1$ are nonnegative decreasing functions on $[0, \infty)$. Set $\theta = \frac1{1+\nu}$, $\delta = f _0 ^{-\theta} f _1 ^{\theta}$, then
\begin{align*}
    2|f| \le f _0 ^{-\theta} f _1 ^{\theta} f _0 + f _0 ^{\nu \theta} f _1 ^{-\nu \theta} f _1 = 2 f _0 ^{1 - \theta} f _1 ^\theta.
\end{align*}
Then
\begin{align*}
    \| f \| _{L ^{p, q}} &= \| \lambda ^{\frac1p - \frac1q} f (\lambda) \| _{L ^q} \\
    &\le \| 
    \lambda ^{\frac{1-\theta}{p _0}-\frac{1-\theta}{q _0}} f _0 ^{1-\theta} (\lambda) \cdot
    \lambda ^{\frac\theta{p _1}-\frac \theta{q _1}} f _1 ^\theta (\lambda) \| _{L ^q} \\
    &\le \| \lambda ^{\frac{1-\theta}{p _0}-\frac{1-\theta}{q _0}} f _0 ^{1-\theta} (\lambda) \| _{L ^{\frac {q _0}{1 - \theta}}} \| \lambda ^{\frac\theta{p _1}-\frac \theta{q _1}} f _1 ^\theta (\lambda) \| _{L ^\frac{q _1}\theta} \\
    &= \| \lambda ^{\frac 1{p_0}-\frac 1{q _0}} f _0 \| _{L ^{q _0}} ^{1 - \theta} \| \lambda ^{\frac 1{p _1}-\frac 1{q _1}} f _1 \| _{L ^{q _1}} ^\theta \\
    &= \|f _0\| _{L ^{p _0, q _0}} ^{1 - \theta} \|f _1\| _{L ^{p _1, q _1}} ^{\theta}.
\end{align*}
\end{proof}

We would also like to mention that Riesz transform is bounded on Lorentz space. The proof can be found in \cite{Carrillo2007}. See \cite{Sawyer1990} for general Lorentz spaces.

\begin{lemma}
\label{lem:riesz-lorentz}
For $1 < p < \infty$, $1 \le q \le \infty$, $R _{ij} = \partial _i \partial _j \La \inv$ is a bounded linear operator from $L ^{p,q} (\R ^n)$ to itself. As a spatial operator, it is also bounded in time-space from $L ^{p,q} ((0, T) \times \R ^n)$ to itself.
\end{lemma}

\subsection{Helmholtz decomposition}

First recall two vector calculus identities:
\begin{align}
    \label{vc1}
    \grad (u \cdot v) &= (u \cdot \grad) v + (v \cdot \grad) u + u \cross \curl v + v \cross \curl u, \\
    \label{vc3}
    \curl (u \cross v) &= u \div v - v \div u + (v \cdot \grad) u - (u \cdot \grad) v.
\end{align}
For operators $A$ and $B$, denote $[A, B] = AB - BA$ to be their commutator. Define $\Pcurl = -\curl \curl \La \inv$ and $\Pgrad = \grad \La \inv \div = \Id - \Pcurl$ to be the Helmholtz decomposition. Then we compute the following commutators.
\begin{align}
    \label{cm1}
    [\vp, \curl] u &= -\grad \vp \times u, \\
    \label{cm2}
    [\vp, \La] u &= -2 \grad \vp \cdot \grad u - (\La \vp) u = -2 \div (\grad \vp \tensor u) + (\La \vp) u, \\
    \label{cm3}
    [\vp, \La \inv] u &= \La \inv \left\{2 \grad \vp \cdot \grad \La \inv u + (\La \vp) \La \inv u\right\}, \\
    \label{cm4}
    [\vp, \Pcurl] u &= \grad \vp \cross \curl \La \inv u + 
        \grad \vp \div \La \inv u 
        - \La \inv u \La \vp \\
    \notag
    & \qquad + (\La \inv u \cdot \grad) \grad \vp 
        - (\grad \vp \cdot \grad) \La \inv u
    \\
    \notag 
    &\qquad + \Pcurl \left\{2 \grad \vp \cdot \grad \La \inv u + (\La \vp) \La \inv u\right\}.
\end{align}
The first two are straightforward. The third uses 
$$[\vp, \La \inv] = -\La \inv [\vp, \La] \La \inv,$$
and the last one is because
\begin{align*}
    [\vp, \Pcurl] &= [\vp, -\curl \curl \La \inv] \\
    &= -[\vp, \curl] \curl \La \inv - \curl [\vp, \curl] \La \inv - \curl \curl [\vp, \La \inv], \\
    [\vp, \Pcurl] u &= \grad \vp \cross \curl \La \inv u + \curl (\grad \vp \cross \La \inv u) \\
    & \qquad - \curl \curl \La \inv \left\{2 \grad \vp \cdot \grad \La \inv u + (\La \vp) \La \inv u\right\},
\end{align*}
and we can expand $\curl (\grad \vp \cross \La \inv u)$ by \eqref{vc3}.

\begin{lemma}
\label{lem-commutator}
$\partial _i [\vp, \Pcurl]$ and $[\vp, \Pcurl] \partial _i$ are both bounded linear operator from $L ^p$ to $L ^p$ for any $1 < p < \infty$, i.e.
\begin{align*}
    \| \partial _i [\vp, \Pcurl] u \| _{L ^p} + \| [\vp, \Pcurl] \partial _i u \| _{L ^p} \le C _{p, \vp} \| u \| _{L ^p}.
\end{align*}
\end{lemma}

\begin{proof}
First, we observe that by Jacobi identity $[\vp, \Pcurl] \partial _i$ and $\partial _i [\vp, \Pcurl]$ differ by
\begin{align*}
    [[\vp, \Pcurl], \partial _i] = [\vp, [\Pcurl, \partial _i]] - [\Pcurl, [\vp, \partial _i]] = 0 - [\Pcurl, \partial _i \vp]
\end{align*}
which is bounded from $L ^p$ to $L ^p$ for any $p$, because both $\Pcurl$ and multiplication by $\partial _i \vp$ are bounded from $L ^p$ to $L ^p$, so we can complete the proof by duality. For $1 < p < 3$, set $\frac1{p ^*} = \frac1p - \frac13$, from \eqref{cm4} we can see
\begin{align*}
    \| [\vp, \Pcurl] \partial _i u \| _{L ^p} & \lesssim
    \| \grad \La \inv \partial _i u \| _{L ^{p} (\R ^3)} + C _{p, \vp} 
    \| \La \inv \partial _i u \| _{L ^{p} (\supp \vp)} \\ 
    & \lesssim 
    \| u \| _{L ^{p} (\R ^3)} + C _{p, \vp} 
    \|\partial _i \La \inv u \| _{L ^{p ^*} (\supp \vp)} \\
    & \le C \| u \| _{L ^p (\R ^3)}.
\end{align*}
For $\frac32 < p < \infty$, set $1 - \frac1p = \frac1q = \frac1{q ^*} + \frac13$, then $1 < p, q, q^* < \infty$. Take any $u \in L ^p (\R ^3)$ and any vector field $v \in L ^q (\R ^3)$, 
\begin{align*}
    \int \partial _i [\vp, \Pcurl] u \cdot v \d x 
    & = -\int [\vp, \Pcurl] u \cdot \partial _i v \d x \\
    & = \int u \cdot [\vp, \Pcurl] \partial _i v \d x \\
    & \le \| u \| _{L ^p (\R ^3)} (\| v \| _{L ^{q} (\R ^3)} + \|\partial _i \La \inv v\| _{L ^{q} (\supp \vp)}) \\
    & \le \| u \| _{L ^p (\R ^3)} (\| v \| _{L ^{q} (\R ^3)} + C _{q, \vp} \|\partial _i \La \inv v\| _{L ^{q ^*} (\supp \vp)}) \\
    & \le C \| u \| _{L ^p (\R ^3)} \| v \| _{L ^{q} (\R ^3)}.
\end{align*}
\end{proof}

\begin{corollary}
$\partial _i [\vp, \Pgrad]$ and $[\vp, \Pgrad] \partial _i$ are both bounded linear operator from $L ^p$ to $L ^p$ for any $1 < p < \infty$.
\end{corollary}

\begin{proof}
$\Id = \Pgrad + \Pcurl$ commutes with $\vp$, so $[\vp, \Pgrad] = -[\vp, \Pcurl]$.
\end{proof}

Because of the smoothing effect of the Laplace potential, we have the following.

\begin{lemma}
\label{lem:smoothing}
    Let $\vp \in \cci (\R ^3)$ be supported away from some openset $\Omega \subset \R ^3$, that is, $\operatorname{dist} (\supp \vp, \Omega) = d > 0$. Then for any $f \in L ^1 _{\loc} (\R ^3)$, $k > 0$,
    \begin{align*}
        \| \La \inv (\vp f) \| _{C ^k (\Omega)} \lesssim _{k, d} \| f \| _{L ^1 (\supp \vp)}.
    \end{align*}
    We also have
    \begin{align*}
        \| \Pgrad (\vp f) \| _{C ^k (\Omega)}, \| \Pcurl (\vp f) \| _{C ^k (\Omega)} \lesssim _{k, d} \| f \| _{L ^1 (\supp \vp)}.
    \end{align*}
\end{lemma}

\section{Proof of the Main Results}

In this section, we show that the Local Theorem \ref{thm:local} leads to the main results. First, we show the pivot quantity is indeed enough to bound $\grad ^n \omega$. 

\begin{lemma}
    \label{lem:pivot}
    There exists $\eta _2 > 0$ such that the following holds. Let $\frac{11}6 < p < 2$, $\frac{2-p}{p-1} < \nu \le \frac{7p-12}{6-p}$. If $u$ is a suitable solution to the Navier-Stokes equations in $(-9, 0) \times \R ^3$ satisfying the following conditions,
    \begin{align}
        \label{eqn:mean-zero-property-2}
        &\int _{B _1} u (t, x) \phi (x) \d x = 0, \qquad \text{\upshape a.e. } t \in (-9, 0), \\
        \label{eqn:Lp-smallness}
        & \delta ^{-\nu} \left(\int _{Q _3} |\grad u| ^p \d x \d t\right) ^\frac1p \le \eta _2, \\
        & \delta \int _{Q _3} |\grad u| ^{2} \d x \d t \le \eta _2,
    \end{align}
    for some $\delta \le \eta _2$, 
    then we have for any $n \ge 0$,
    \begin{align*}
        \| \grad ^n \omega \| _{L ^\infty _{t, x} (Q _{8 ^{-n-2}})} \le C _n.
    \end{align*}
    Here $C _n$ is the same constant in Theorem \ref{thm:local}.
\end{lemma}

\begin{proof}
First, we claim that 
\begin{align}
\label{eqn:LinfL1}
    \delta \| \omega \| _{L ^\infty (-4, 0; L ^1 (B _2))} \le C \eta _2.
\end{align}
Formally, we can take the dot product of both sides of the vorticity equation \eqref{eqn:ns} with $\omega ^0 := \frac{\omega}{|\omega|}$, and recalling the convexity inequality $\omega ^0 \cdot \La \omega \le \La |\omega|$, we have
\begin{align}
\label{eqn:|omega|}
    (\pt + u \cdot \grad - \La) |\omega| - \omega \cdot \grad u \cdot \omega^ 0 \le 0.
\end{align}
Let $\psi \in \cci ((-9, 0] \times \R ^3)$ be a cut-off function such that $\ind{Q _2} \le \psi \le \ind{Q _3}$. Multiply \eqref{eqn:|omega|} by $\psi$ and then integrate in space,
\begin{align*}
\frac{\d}{\d t} \int \psi |\omega| \d x
    &\le \int \left[
        (\pt + u \cdot \grad + \La) \psi
    \right] |\omega| \d x 
    + 
    \int \psi \omega \cdot \grad u \cdot \omega ^0 \d x \\
    &\le C \int _{B _3} 1 + |u| ^2 + |\grad u| ^2 \d x \le C \left(1 + \int _{B _3} |\grad u| ^2 \d x\right).
\end{align*}
for some large universal constant $C > 1$. The last step uses Poincar\'e's inequality and \eqref{eqn:mean-zero-property-2}. Integrate in time we obtain 
\begin{align*}
    \| \omega \| _{L ^\infty (-4, 0; L ^1 (B _2))} \le C \left(1 + \frac{\eta _2}\delta \right) \le 2 C \frac{\eta _2}\delta.
\end{align*}
This proves the claim. A more rigorous proof can be obtained by difference quotient same as in Constantin \cite{Constantin1990} or Lions \cite{Lions1996} Theorem 3.6, so we omit the details.

Now we interpolate between \eqref{eqn:Lp-smallness} and \eqref{eqn:LinfL1}. 
Let $\theta = \frac1{1+\nu}$,
\begin{align*}
    \| \omega \| _{\pq2 (Q _2)} \le
    \| \omega \| _{L ^p (Q _2)} ^\theta 
    \| \omega \| _{L ^\infty _t L ^1 _x (Q _2)} ^{1 - \theta} \le (2 C) ^{1-\theta} \eta _2 \delta ^{\theta \nu + \theta - 1} \le \frac12 \eta _1,
\end{align*}
where we choose $\eta _2 = \frac{\eta _1}{4C + 1} \le \frac12 \eta _1$ from Theorem \ref{thm:local}, 
and $p _2, q _2$ are determined by
\begin{align*}
    \frac1{p _2} = \frac\theta p, \qquad \frac1{q _2} = \frac\theta p + 1 - \theta.
\end{align*}
Combine the above with \eqref{eqn:Lp-smallness} we have
\begin{align}
    \label{eqn:local-requirement2}
    \| \grad u \| _{L ^p _t L ^p _x (Q _2)}  + \| \omega \| _{\pq2 (Q _2)} \le \frac12 \eta _1 + \frac12 \eta _1 = \eta _1.
\end{align}
By the choice of $\theta$ and the range of $\nu$,
\begin{align*}
    \frac1p + \frac1{p _2} &= \frac1p + \frac1{p (1 + \nu)} = \frac{2 + \nu}{p (1 + \nu)} < 1, \\
    \frac1p + \frac1{q _2} &= \frac1p + \frac{1 + \nu p}{p (1 + \nu)} = \frac{2 + \nu + \nu p}{p (1 + \nu)} \le \frac76.
\end{align*}
One can also easily check that $p < 2$ implies $q _2 < 2$, and thus by \eqref{eqn:mean-zero-property-2} and \eqref{eqn:local-requirement2} the requirements of the Local Theorem \ref{thm:local} are satisfied with $p _1 = q _1 = p$, and it completes the proof of the lemma.
\end{proof}





Now we transform this lemma into the global coordinate. Recall that $Q _\e (t, x)$ is defined by \eqref{eqn:def-Qe}.

\begin{corollary}
\label{cor:eta3}
There exists $\eta _3 > 0$ such that the following holds. If for some $\delta \le \eta _2$,
\begin{align}
    \label{eqn:eta2}
    \delta ^{-2\nu} \left(
        \fint _{Q _\e (t, x)} |\grad u| ^p \d x \d t
    \right)^\frac2p 
    + \delta \fint _{Q _\e (t, x)} |\grad u| ^2 \d x \d t \le \eta _3 \e ^{-4},
\end{align}
then
\begin{align*}
    |\grad ^n \omega (t, x)| \le C _n \e ^{-n-2}.
\end{align*}
\end{corollary}

\begin{proof}
Define $\tilde u$ by \eqref{eqn:tildeu}. Then \eqref{eqn:eta2} implies
\begin{align*}
    \delta ^{-2\nu} \left(
        \fint _{Q _3} |\grad \tilde u| ^p \d x \d t
    \right)^\frac2p 
    \le \eta _3, &\qquad \delta \fint _{Q _3} |\grad \tilde u| ^2 \d x \d t \le \eta _3 \\
    \Rightarrow \delta ^{-\nu} \left(
        \int _{Q _3} |\grad \tilde u| ^p \d x \d t 
    \right) ^\frac1p \le \eta _3 ^\frac12 |Q _3| ^\frac1p, &\qquad \delta \int _{Q _3} |\grad \tilde u| ^2 \d x \d t \le \eta _3 |Q _3|.
\end{align*}
Moreover, \eqref{eqn:mean-zero-property-2} is satisfied by $\tilde u$. Therefore, if we choose $\eta _3$ such that 
\begin{align*}
    \max \left\lbrace 
        \eta _3 ^\frac12 |Q _3| ^\frac1p, \eta _3 |Q _3| 
    \right\rbrace = \eta _2,
\end{align*}
then by Lemma \ref{lem:pivot}, $\tilde \omega := \curl \tilde u$ has bounded derivatives at $(0, 0)$, and thus finish the proof of the corollary by scaling.
\end{proof}


Then we use the maximal function to go from the local bound to a global bound.

\begin{proof}[Proof of Theorem \ref{thm:main}]
First, we fix $\frac{11}6 < p < 2$, $\frac{2-p}{p-1} < \nu \le \frac{7p-12}{6-p}$. Let $\eta << 1$ be a small constant to be specified later. Finally we fix a $0 < \delta < \infty$. For $(t, x) \in (0, T) \times \R ^3$, define
\begin{align*}
    I (\e) = \e ^4 \left[ 
        \delta ^{-2\nu} \left(
            \fint _{Q _{\e} (t, x)} |\mm (\grad u)| ^p
        \right)^\frac2p 
        + \delta \fint _{Q _\e(t, x)} |\mm (\grad u)| ^2
    \right].
\end{align*}
If $(t, x)$ is both a $\mathcal Q$-Lebesgue point of $|\mm (\grad u)| ^p$ and of $|\mm (\grad u)| ^2$, then we claim that there exists a positive $\e = \e _{(t, x)}$ such that one of the two cases is true:
\begin{enumerate}[{\ttfamily C{a}se 1.}]
    \item $3 \e _{(t, x)} < t ^\frac12$, and $I (\e _{(t, x)}) = \eta$.
    \item $3 \e _{(t, x)} = t ^\frac12$, and $I (\e _{(t, x)}) \le \eta$.
\end{enumerate}
This is because by Theorem \ref{thm:lebesgue}
\begin{align*}
    \lim _{\e \to 0} I (\e) = 0 ^4 \left[ 
        \delta ^{-2 \nu} (|\mmu (t, x)| ^p) ^\frac2p + \delta |\mmu (t, x)| ^2
    \right] = 0,
\end{align*}
and $I (\e)$ is a continuous function of $\e$. 

One the one hand, in both cases we have $I (\e) \le \eta$, which implies that
\begin{align*}
    \delta ^{-\nu} \e ^2 \left(
        \fint _{Q _{\e} (t, x)} |\mm (\grad u)| ^p
    \right)^\frac1p \le \sqrt \eta, \qquad 
    \delta ^\frac12 \e ^2 \left(
        \fint _{Q _\e(t, x)} |\mm (\grad u)| ^2
    \right) ^\frac12 \le \sqrt \eta.
\end{align*}
If we set $\eta < \eta _0 ^2$, then depending on $\delta \ge 1$ or $\delta \le 1$, one of the two would imply admissibility condition \eqref{eqn:admissibility} by Jensen's inequality. Therefore $Q _\e (t, x)$ is admissible and
\begin{align*}
    I (\e) \le \e ^4 \left[
        \delta ^{-2\nu} \mmq (\mm(\grad u) ^p) ^\frac2p + \delta \mmq (\mm (\grad u) ^2)
    \right],
\end{align*}
so we can combine two canses and conclude
\begin{align}
    \label{eqn:e-4}
    \e ^{-4} _{(t, x)} \le \max \left\lbrace 
        \frac1\eta \left[ 
            \delta ^{-2\nu} \mmq (\mm(\grad u) ^p) ^\frac2p + \delta \mmq (\mm (\grad u) ^2)
        \right], 81 t ^{-2}
    \right\rbrace.
\end{align}


On the other hand, if we set $\eta < \eta _3$, then in both cases $I (\e) \le \eta _3$. If $\delta \le \eta _2$ one would have
\begin{align}
    \label{eqn:dnomega}
    |\grad ^n \omega (t, x)| \le C _n \e ^{-n-2}
\end{align}
by Corollary \ref{cor:eta3}. If $\delta > \eta_2$, notice that by Jensen's inequality, 
\begin{align*}
    \left(
            \fint _{Q _{\e} (t, x)} |\mm (\grad u)| ^p
        \right)^\frac2p 
    \le \fint _{Q _\e(t, x)} |\mm (\grad u)| ^2,
\end{align*}
so
\begin{align*}
    I (\e) &\ge \e ^4 \left[ 
        (\delta ^{-2\nu} + \delta - \eta _2) \left(
            \fint _{Q _{\e} (t, x)} |\mm (\grad u)| ^p
        \right)^\frac2p 
        + \eta _2 \fint _{Q _\e(t, x)} |\mm (\grad u)| ^2
    \right] \\
    &\ge \e ^4 \left[ 
        (1 - \eta _2) \left(
            \fint _{Q _{\e} (t, x)} |\mm (\grad u)| ^p
        \right)^\frac2p 
        + \eta _2 \fint _{Q _\e(t, x)} |\mm (\grad u)| ^2
    \right] \\
    &\ge (1 - \eta _2) \eta _2 ^{2 \nu} \e ^4 \left[ 
        \eta _2 ^{-2\nu} \left(
            \fint _{Q _{\e} (t, x)} |\mm (\grad u)| ^p
        \right)^\frac2p 
        + \eta _2 \fint _{Q _\e(t, x)} |\mm (\grad u)| ^2
    \right].
\end{align*}
If we require $\eta < (1 - \eta _2) \eta _2 ^{2 \nu} \eta _3$, then
\begin{align*}
    \e ^4 \left[ 
        \eta _2 ^{-2\nu} \left(
            \fint _{Q _{\e} (t, x)} |\mm (\grad u)| ^p
        \right)^\frac2p 
        + \eta _2 \fint _{Q _\e(t, x)} |\mm (\grad u)| ^2 
    \right] \le \eta _3.
\end{align*}
again by Corollary \ref{cor:eta3} we would still have \eqref{eqn:dnomega}. In conclusion, we choose
\begin{align*}
    \eta = \min \left\lbrace
        \eta _0 ^2, (1 - \eta _2) \eta _2 ^{2 \nu} \eta _3
    \right\rbrace ,
\end{align*}
then for any $0 < \delta < \infty$ one would have
\begin{align*}
    |\grad ^n \omega (t, x)| ^\frac{4}{n+2} \le C _n ^\frac{4}{n+2} \max \left\lbrace 
        \frac1\eta \left[ 
            \delta ^{-2\nu} \mmq (\mm(\grad u) ^p) ^\frac2p + \delta \mmq (\mm (\grad u) ^2)
        \right], 81 t ^{-2}
    \right\rbrace
\end{align*}
by putting \eqref{eqn:dnomega} and \eqref{eqn:e-4} together. Denote $f = |\grad ^n \omega| ^\frac{4}{n+2}$, and we denote $f _1 = \mmq (\mm(|\grad u|) ^p) ^\frac2p$, $f _2 = \mmq (\mm (|\grad u|) ^2)$. Then we have almost everywhere
\begin{align*}
    f \inds{f > C _n t ^{-2}} \lesssim _n \delta ^{-2 \nu} f _1 + \delta f _2.
\end{align*}
By Theorem \ref{thm:maximal-function},
\begin{align*}
    \| f _1 \| _{L ^1} &\le C _p \| \mm (\grad u) ^p \| _{L ^\frac2p} ^\frac2p \lesssim C _p \|\mm (\grad u) ^2 \| _{L ^1} \le C _p \|\grad u\| _{L ^2} ^2, \\
    \| f _2 \| _{L ^{1, \infty}} &\le C _1 \| \mm (\grad u) ^2 \| _{L ^1} \le C _1 \|\grad u\| _{L ^2} ^2.
\end{align*}
Finally, by the interpolation between Lorentz spaces Lemma \ref{lem:lorentz}, 
\begin{align*}
    \| f \inds{f > C _n t ^{-2}} \| _{L ^{1, 1 + 2 \nu}} \lesssim _{p, n} \| \grad u \| _{L ^2 ((0, T) \times \R ^3)} ^2 \le \| u _0 \| _{L ^2 (\R ^3)} ^2.
\end{align*}
This proves the theorem for $q \ge 1 + 2 \nu$. Recall that 
$p$ can be arbitrarily chosen between $\frac{11}6$ and $2$, and $\nu$ can be chosen between $\frac{2-p}{p-1}$ and $\frac{7p-12}{6-p}$, so $\nu$ can be arbitrarily small, therefore we prove the theorem for any $q > 1$.
\end{proof}

Estimates on $\grad ^2 u$ can be obtained by a Riesz transform of $\La u = -\curl \omega$.

\begin{proof}[Proof of Corollary \ref{cor:L43}]
We can put $K \subset (t _0, T) \times B _R$ for some $t _0, T, R > 0$. Denote $Q = (t _0, T) \times B _{2R}$. Let $\rho \in C _c ^\infty (\R ^3)$ be a smooth spatial cut-off function between $\ind{B _R} \le \rho \le \ind{B _{2R}}$. Then 
\begin{align*}
    &\| \La (\rho u) \| _{L ^{\frac43,q} ((t _0, T) \times \R ^3)}
    \lesssim _\rho \| \La u \| _{L ^{\frac43,q} (Q)} 
    + \| \grad u \| _{L ^{\frac43,q} (Q)}
    + \| u \| _{L ^{\frac43,q} (Q)}.
\end{align*}
Since $\La u = -\curl \omega$, the case $n = 1$ of Theorem \ref{thm:main} gives
\begin{align*}
    \| \La u \inds{|\La u| > C _1 t ^{-\frac32}} \| _{L ^{\frac43, q} ((0, T) \times \R ^3)} \le C _q \| u _0 \| _{L ^2 (\R ^3)} ^\frac32,
\end{align*}
so
\begin{align*}
    \| \La u \| _{L ^{\frac43,q} (Q)} \le C _q \| u _0 \| _{L ^2 (\R ^3)} ^\frac32 + C _1 \| t ^{-\frac32} \| _{L ^{\frac43} (Q)} \lesssim C _q \| u _0 \| _{L ^2 (\R ^3)} ^\frac32 + C _1 \left(\frac{R ^3}{t _0}\right)^\frac34.
\end{align*}
As for lower order terms,
\begin{align*}
    \| \grad u \| _{L ^\frac43 (Q)} &\lesssim \| \grad u \| _{L ^2 (Q)}, \\
    \| u \| _{L ^\frac43 (Q)} &\le \| u \| _{L ^\infty _t L ^2 _x (Q)}.
\end{align*}
For Leray-Hopf solution, $\|\grad u\| _{L ^\infty _t L ^2 _x \cap L ^2 _t \dot H ^1 _x ((0, T) \times \R ^3)} \le \|u _0\| _{L ^2}$, so
\begin{align*}
    \| \La (\rho u) \| _{L ^{\frac43,q} ((t _0, T) \times \R ^3)} &\lesssim _{q, K} 
    \| u _0 \| _{L ^2 (\R ^3)} ^\frac32 + 1 + 
    \| u _0 \| _{L ^2 (\R ^3)} \lesssim \| u _0 \| _{L ^2 (\R ^3)} ^\frac32 + 1.
\end{align*}
Because Riesz transform is bounded from $L ^{\frac43, q} ((t _0, T) \times \R ^3)$ to itself by Lemma \ref{lem:riesz-lorentz}, 
\begin{align*}
    \| \grad ^2 u \| _{L ^{\frac43,q} (K)} \le \| \grad ^2 (\rho u) \| _{L ^{\frac43,q} (Q)} \lesssim _{q, K} \| u _0 \| _{L ^2 (\R ^3)} ^\frac32 + 1.
\end{align*}
\end{proof}

\begin{remark}
For smooth solutions to the Navier-Stokes equation, we have $L ^{1, q}$ estimate for the third derivatives for any $q > 1$,
\begin{align*}
    \left\|
        \grad ^2 \omega \inds{|\grad ^2 \omega| > C t ^{-2}}
    \right\| _{L ^{1, q} ((0, T) \times \R ^3)} \le C _q \| u _0 \| _{L ^2} ^2.
\end{align*}
\end{remark}

\section{Local Study: Part One, Initial Energy}

The following three sections are dedicated to the proof of the Local Theorem \ref{thm:local}. In \cite{Vasseur2010}, the proof of the local theorem consists of the following three parts:
\begin{enumerate}[{\ttfamily Step 1.}]
    \item Show the velocity $u$ is locally small in the energy space $\mathcal E = L ^{\infty} _t L ^2 _x \cap L ^2 _t H ^1 _x$.
    \item Use De Giorgi iteration and the truncation method developed in \cite{Vasseur2007} to show $u$ is locally bounded in $L ^\infty$.
    \item Bootstrap to higher regularity by differentiating the original equation.
\end{enumerate}
In our case, directly working with $u$ is difficult due to the lack of control on the pressure, which is nonlocal. Therefore, we would like to work on vorticity, whose evolution is governed by \eqref{eqn:vorticity} and only involves local quantities. Since $\omega$ is one derivative of $u$, we have less integrability to do any parabolic regularization, and we don't have the local energy inequality to perform De Giorgi iteration. This motivates us to work on minus one derivative of $\omega$, but instead of $\omega$ we use a localization of $\omega$. Similar as \cite{Chamorro2018}, we introduce a new local quantity
\begin{align*}
\boxed{
    v := -\curl \vps \La \inv \vp \curl u = -\curl \vps \La \inv \vp \omega.
}
\end{align*}
where $\vp$ and $\vps$ are a pair of fixed smooth spatial cut-off functions, which are defined between $\ind{B _\frac65} \le \vp \le \ind{B _\frac54}$, $\ind{B _\frac43} \le \vps \le \ind{B _\frac32}$. This $v$ is divergence free and compactly supported. It will help us get rid of the pressure $P$, while staying in the same space as $u$: it scales the same as $u$, has the same regularity, inherit a local energy inequality from $u$, and its evolution only depends on local information. We will follow the same three steps above, but we will work on $v$ instead of $u$.

For convenience, from now on we will use $\eta$ to denote a small universal constant depending only on the smallness of $\eta _1$, such that $\lim _{\eta _1 \to 0} \eta = 0$. Similar as the constant $C$, the value of $\eta$ may change from line to line. The purpose of this section is to obtain the smallness of $v$ in the energy space $\mathcal E$, which is the following proposition. 
\begin{proposition}
\label{prop:v-L-inf-L-2}
Under the same assumptions of the Local Theorem \ref{thm:local}, we have
\begin{align}
    \label{eqn:v-L-infty-L-2}
    \| v \| _{\mathcal E (Q _1)} ^2 = \sup _{t \in (-1, 0)} \int _{B _1} |v (t)| ^2 \d x + \int _{Q _1} |\grad v| ^2 \d x \le \eta.
\end{align}
\end{proposition}

For convenience, define $q _3$, $q _4$, $q _5$ by
\begin{align*}
    \frac1{q _3} = \frac1{q _1} - \frac13, \qquad \frac1{q _4} = \frac1{q _2} - \frac13, \qquad
    \frac1{q _5} = \left(\frac1{q _3} - \frac13\right) _+.
\end{align*}

\subsection{Equations of v}
We use \eqref{vc1} in \eqref{eqn:ns} to rewrite the equation of $u$, then take the curl to rewrite the equation of $\omega$, finally apply $-\curl \vps \La \inv \vp$ on the vorticity equation to obtain the equation of $v$.
\begin{align}
    \notag 
    \pt u + \Pcurl (\omega \cross u) &= \La u, \\
    \notag
    \pt \omega + \curl (\omega \cross u) &= \La \omega, \\
    \label{eqn-v-original}
    \pt v - \curl \vps \La \inv \vp \curl (\omega \cross u) &= - \curl \vps \La \inv \vp \La \omega.
\end{align}
The second term of \eqref{eqn-v-original} is
\begin{align*}
    \curl \vps \La \inv \vp \curl (\omega \cross u) &= \BB - \Pcurl (\vp \omega \cross u)
\end{align*}
where $\BB$ denotes the quadratic commutator
\begin{align*}
    \BB &:= - \curl (1 - \vps) \La \inv \vp \curl (\omega \cross u) + \curl \La \inv [\vp, \curl] (\omega \cross u) \\
    &= - \curl (1 - \vps) \La \inv \vp \curl (\omega \cross u) + \curl \La \inv (-\grad \vp \cross (\omega \cross u))
\end{align*}
Here we used \eqref{cm1}. The right hand side of \eqref{eqn-v-original} is
\begin{align*}
    -\curl \vps \La \inv \vp \La \omega &= \La v + \BL
\end{align*}
where $\BL$ denotes the linear commutator
\begin{align*}
    \BL &:= 
    [-\curl \vps \La \inv \vp, \La] \omega \\
    &= -\curl [\vps \La \inv \vp, \La] \omega \\
    &= -\curl [\vps, \La] \La \inv \vp \omega -\curl \vps \La \inv [\vp, \La] \omega \\
    &= -\curl [\vps, \La] \La \inv \vp \omega + \curl \vps \La \inv \left(2 \div (\grad \vp \tensor \omega) - (\La \vp) \omega \right).
\end{align*}
Here we used \eqref{cm2}.
Therefore we have the equation for $v$ as the following,
\begin{align}
\label{eqn-v}
    \pt v + \Pcurl (\vp \omega \cross u) = \BB + \BL + \La v.
\end{align}
We observe the following localization decomposition.
\begin{lemma}
We can decompose
\begin{align*}
    \vp u = v + w, \qquad \vp \omega = \curl v + \varpi,
\end{align*}
where $w$ and $\varpi$ are harmonic inside $B _1$.
\end{lemma}

\begin{proof}
We can compute $v$ by
\begin{align*}
    v &= -\curl \vps \La \inv \vp \curl u \\
    &= \curl (1 - \vps) \La \inv \vp \curl u - \curl \La \inv \vp \curl u \\
    &= \curl (1 - \vps) \La \inv \vp \omega -\curl \La \inv [\vp, \curl] u + \Pcurl (\vp u) \\
    &= \curl (1 - \vps) \La \inv \vp \omega + \curl \La \inv (\grad \vp \times u)  - \Pgrad (\vp u) + \vp u \\
    &= \curl (1 - \vps) \La \inv \vp \omega + \curl \La \inv (\grad \vp \times u) - \grad \La \inv (\grad \vp \cdot u) + \vp u
\end{align*}
using $\div u = 0$. We denote
\begin{align*}
    w := -\curl (1 - \vps) \La \inv \vp \omega -\curl \La \inv (\grad \vp \times u) + \grad \La \inv (\grad \vp \cdot u),
\end{align*}
which implies the first decomposition $\vp u = v + w$. By taking the curl, 
\begin{align*}
    \curl (\vp u) &= \curl v + \curl w, \\
    \grad \vp \cross u + \vp \omega 
    &= \curl v - \curl \curl (1 - \vps) \La \inv \vp \omega - \curl \curl \La \inv (\grad \vp \cross u) \\
    &= \curl v - \curl \curl (1 - \vps) \La \inv \vp \omega + \Pcurl (\grad \vp \cross u).
\end{align*}
We denote
\begin{align*}
    \varpi &:= - \curl \curl (1 - \vps) \La \inv \vp \omega - \Pgrad (\grad \vp \cross u) \\
    &= - \curl \curl (1 - \vps) \La \inv \vp \omega - \grad \La \inv \div (\grad \vp \cross u) \\
    &= - \curl \curl (1 - \vps) \La \inv \vp \omega + \grad \La \inv (\grad \vp \cdot \omega)
\end{align*}
which implies the second decomposition $\vp \omega = \curl v + \varpi$. We can easily see that $\La w$ and $\La \varpi$ are both the sum of a smooth function supported outside $B _\frac32$ and the Newtonian potential of something supported inside $\operatorname{supp} (\grad \vp) \subset B _{\frac54} \setminus B _{\frac65}$, so they are harmonic inside $B _1$.
\end{proof}
Using this decomposition, we can continue to expand
\begin{align*}
    \Pcurl (\vp \omega \cross u) &= \vp \omega \cross u - \Pgrad(\vp \omega \cross u) \\
    &= \omega \cross v + \omega \cross w - \frac12 \Pgrad \left(
        (\curl v + \varpi) \cross u + \omega \cross (v + w)
    \right) \\
    &= \omega \cross v - \frac12 \Pgrad \left(
        \curl v \cross u + \omega \cross v
    \right) - \BW,
\end{align*}
where $\BW$ denotes the remainders involving $w$ and $\varpi$,
\begin{align*}
    \BW := -\omega \cross w + \frac12 \Pgrad \left(
        \varpi \cross u + \omega \cross w
    \right).
\end{align*}
By subtracting \eqref{vc3} from \eqref{vc1}, for divergence free $u$, $v$ we have
\begin{align*}
    \curl v \cross u + \curl u \cross v = -\grad (u \cdot v) + 2 u \cdot \grad v + \curl (u \cross v),
\end{align*}
so 
\begin{align*}
    \Pcurl (\vp \omega \cross u) &= \omega \cross v + \frac12 \grad (u \cdot v) - \Pgrad \div (u \tensor v) - \BW \\
    &= \omega \cross v + \grad \left(
        \frac12 u \cdot v - \La \inv \div \div (u \tensor v)
    \right)
    - \BW.
\end{align*}
\newcommand{\tr}{\operatorname{tr}}
For convenience, denote the Riesz operator
\begin{align*}
    \RR = \frac12 \tr - \La \inv \div \div
\end{align*}
Finally, we have the equation of $v$ as
\begin{align}
    \label{eqn-v-final}
    \pt v + \omega \cross v + \grad \RR (u \tensor v) = \BB + \BL + \BW + \La v, \qquad \div v = 0.
\end{align}

We now check the spatial integrability of these new terms. 
\begin{lemma}
    \label{lem:integrability-1}
    For any $1 < p < \infty$, 
    \begin{align*}
        &\| v \| _{L ^p}, \| \grad w \| _{L ^p}, \| \varpi \| _{L ^p} \lesssim \| \omega \| _{L ^1 (B _2)} + \| u \| _{L ^p (B _2)}, \\
        &\| \grad v \| _{L ^p}, \| \grad \varpi \| _{L ^p} \lesssim \| \omega \| _{L ^p (B _2)}, \\
        &\| \grad ^2 w \| _{L ^p} \lesssim \| u \| _{W ^{1,p}} (B _2).
    \end{align*}
    If we denote $q = \min \{1, \frac1p + \frac13\} _+ \inv$, then 
    \begin{align*}
        \| \BB \| _{L ^p (B _2)} &\lesssim \| \omega \cross u \| _{L ^q (B _2)}, \\
        \| \BL \| _{L ^p (B _2)} &\lesssim \| \omega \| _{L ^p (B _2)} , \\
        \| \BW \| _{L ^p (B _2)} &\lesssim \| \omega \cross w \| _{L ^p (B _2)} + \| \varpi \cross u \| _{L ^p (B _2)} .
    \end{align*}
\end{lemma}

\begin{proof}
$v, w, \varpi$ are all supported inside $B _2$, so
\begin{align*}
    \| v \| _{L ^p} &\le \| \vp u \| _{L ^p} + \| w \| _{L ^p} \lesssim \| u \| _{L ^p (B _2)} + \| \grad w \| _{L ^p}, \\
    \| \grad w \| _{L ^p} &\le \| \grad \curl (1 - \vps) \La \inv \vp \omega \| _{L ^p (B _2)} 
    + \| \grad \curl \La \inv (\grad \vp \cross u) \| _{L ^p} \\
    &\qquad + \| \grad ^2 \La \inv (\grad \vp \cdot u) \| _{L ^p} \\
    & \lesssim \| (1 - \vps) \La \inv \vp \omega \| _{C ^2} + \| \grad \vp \cross u \| _{L ^p} + \| \grad \vp \cdot u \| _{L ^p} \\
    & \lesssim \| \omega \| _{L ^1 (B _2)} + \| u \| _{L ^p (B _2)}, \\
    \| \varpi \| _{L ^p} &\le \| \curl \curl (1 - \vps) \La \inv \vp \omega \| _{L ^p (B _2)} + \| \grad \La \inv (\grad \vp \cdot \omega) \| _{L ^p} \\
    & \le \| (1 - \vps) \La \inv \vp \omega \| _{C ^2} + \| \grad \La \inv \div (\grad \vp \cross u) \| _{L ^p} \\
    & \le \| \omega \| _{L ^1 (B _2)} + \| \grad \vp \cross u \| _{L ^p} \\
    & \le \| \omega \| _{L ^1 (B _2)} + \| u \| _{L ^p (B _2)}.
\end{align*}
Here we used Lemma \ref{lem:smoothing} since $\vp$ and $1 - \vps$ are supported away from each other, and we also used the boundedness of Riesz transform by Lemma \ref{lem:riesz-lorentz}. Their derivatives are bounded by
\begin{align*}
    \| \grad v \| _{L ^p} 
    &= \| \grad \curl \vps \La \inv \vp \omega \| _{L ^p} \\
    &\le \| \grad \curl \La \inv \vp \omega \| _{L ^p} + \| \grad \curl (1 - \vps) \La \inv \vp \omega \| _{L ^p (B _2)} \\
    &\lesssim \| \omega \| _{L ^p (B _2)} + \| \omega \| _{L ^1 (B _2)} \lesssim \| \omega \| _{L ^p (B _2)}, \\
    \| \grad ^2 w \| _{L ^p} &\le \| \grad ^2 \curl (1 - \vps) \La \inv \vp \omega \| _{L ^p (B _2)} 
    + \| \grad ^2 \curl \La \inv (\grad \vp \cross u) \| _{L ^p} \\
    &\qquad + \| \grad ^3 \La \inv (\grad \vp \cdot u) \| _{L ^p} \\
    & \lesssim \| \omega \| _{L ^1 (B _2)} + \| u \| _{W ^{1, p} (B _2)} \lesssim \| u \| _{W ^{1, p} (B _2)}, \\
    \| \grad \varpi \| _{L ^p} &\le \| \grad \curl \curl (1 - \vps) \La \inv \vp \omega \| _{L ^p (B _2)} + \| \grad ^2 \La \inv (\grad \vp \cdot \omega) \| _{L ^p} \\
    & \lesssim \| \omega \| _{L ^1 (B _2)} + \| \omega \| _{L ^p (B _2)} \lesssim \| \omega \| _{L ^p (B _2)}.
\end{align*}
The proof for $\BB, \BL, \BW$ are similar so we omit here.
\end{proof}

Since $u \in \mathcal E$, it can be seen from the above lemma that $v, \grad w, \varpi \in \mathcal E$, thus
\begin{align*}
    \| \BB \| _{L ^\frac32 (B _2)} &\lesssim \| \omega \cross u \| _{L ^1 (B _2)} \in L ^2 _t, \\
    \| \BL \| _{L ^2 (B _2)} &\lesssim \| \omega \| _{L ^2 (B _2)} \in L ^2 _t, \\
    \| \BW \| _{L ^\frac32 (B _2)} &\lesssim \| \omega \cross w \| _{L ^\frac32 (B _2)} + \| \varpi \cross u \| _{L ^\frac32 (B _2)} \in L ^2 _t,
\end{align*}
therefore $\BB, \BL, \BW \in L ^2 _t L ^\frac32 _{\loc,x}$. In the appendix we prove the suitability for $v$: it satisfies the following local energy inequality,
\begin{align}
\label{local-inequality}
    \pt \frac{|v| ^2}{2} + |\grad v| ^2 + \div \left[v
    \RR (u \tensor v)
    \right] 
    &\le \La \frac{|v| ^2}{2} + v \cdot (\BB + \BL + \BW).
\end{align}


\subsection{Energy Estimate}
Multiply \eqref{local-inequality} by $\vp ^4$ then integrate over $\R ^3$ yields
\begin{align*}
    &\frac{\d}{\d t} \int \vp ^4 \frac{|v| ^2}2 \d x 
    + \int \vp ^4 |\grad v| ^2 \d x \\
    &\le \int \frac{|v| ^2}2 \La \vp ^4 \d x
    + \int (v \cdot \grad \vp ^4) \RR (u \tensor v) \d x
    \\
    &\qquad + \int \vp ^4 v \cdot \BB \d x 
    + \int \vp ^4 v \cdot \BL \d x 
    + \int \vp ^4 v \cdot \BW \d x .
\end{align*}
Let us discuss these terms. For the first four terms on the right hand side,
\begin{align}
    I _\La &:= \int \frac{|v| ^2}2 \La \vp ^4 \d x
    \le C \|\vp ^2 v\| _{L ^2} \|v\| _{L ^2},
    \label{Ila}\\
    I _\RR &:= \int (v \cdot \grad \vp ^4) \RR (u \tensor v) \d x 
    \le C \|\vp ^2 v\| _{L ^2} \|\RR (u \tensor v)\| _{L ^2} 
    \label{IR} \\
    \notag
    &\le C \|\vp ^2 v\| _{L ^2} \|u \tensor v\| _{L ^2} ,\\
    I _\BB &:= \int \vp ^4 v \cdot \BB \d x \le \| \vp ^2 v \| _{L ^2} \| \vp ^2 \BB \| _{L ^2} 
    \label{IB} \\
    \notag
    & \le C \| \vp ^2 v \| _{L ^2} \| \omega \cross u \| _{L ^\frac65 (B _2)},
    \\
    I _\BL &:= \int \vp ^4 v \cdot \BL \d x \le \| \vp ^\frac23 |v| ^\frac13 \| _{L ^6} \| |v| ^\frac23 \| _{L ^{q _3}} \| \vp ^2 \BL \| _{L ^{q _2}} 
    \label{IL} \\
    \notag
    &\le \| \vp ^2 v \| _{L ^2} ^\frac13 
    \| v \| _{L ^{q _3}} ^\frac23
    \| \omega \| _{L ^{q _2} (B _2)}.
\end{align}
Here we use H\"older's inequality, $\vp$ is compactly supported in $B _2$ and
$
    \frac1{q_2} + \frac1{q _3} + \frac16 \le 1
$.
For the $\BW$ term,
\begin{align*}
    I _\BW &:= \int \vp ^4 v \cdot \BW \d x \\
    &= -\int \vp ^4 v \cdot \omega \cross w \d x + \frac12 \int \vp ^4 v \cdot  \Pgrad \left(
        \varpi \cross u + \omega \cross w
    \right) \d x \\
    &= -I _{\BW1} + \frac12 I _{\BW2}.
\end{align*}
For the first one, we break it as
\begin{align*}
    I _{\BW1} = \int \vp ^4 v \cdot \omega \cross w \d x &= \int \vp ^3 v \cross \curl v \cdot w \d x + \int \vp ^3 v \cdot \varpi \cross w \d x.
\end{align*}
Using \eqref{vc1},
\begin{align*}
    v \cross \curl v = \frac12 \grad |v| ^2 - (v \cdot \grad) v,
\end{align*}
we have
\begin{align*}
    \int \vp ^3 v \cross \curl v \cdot w \d x &= -\frac12 \int |v| ^2 \div (\vp ^3 w) \d x + \int v \cdot \grad (\vp ^3 w) \cdot v \d x \\
    &\le C \|\vp ^2 v\| _{L ^2} \left( 
        \|\grad w \tensor v\| _{L ^2} + \|w \tensor v\| _{L ^2}
    \right).
\end{align*}
The remaining is of lower order,
\begin{align*}
    \int \vp ^3 v \cdot \varpi \cross w \d x \le C \|\vp ^2 v\| _{L ^2} \|\varpi \cross w\| _{L ^2} 
\end{align*}
For the second one, 
\begin{align*}
    I _{\BW2} &= \int \Pgrad (\vp ^4 v) \cdot \left(
        \varpi \cross u + \omega \cross w
    \right) \d x\\
    &\le \| \Pgrad (\vp ^4 v) \| _{L ^6} \|\varpi \cross u + \omega \cross w \| _{L ^\frac65}
\end{align*}
where
\begin{align*}
    \| \Pgrad (\vp ^4 v) \| _{L ^6} &= 
    \| [\Pgrad, \vp ^2] \vp ^2 v \| _{L ^6} \lesssim \| \vp ^2 v\| _{L ^2}.
\end{align*}
So $I _\BW$ can be bounded by
\begin{align}
    I _\BW \le C \|\vp ^2 v\| _{L ^2} \left(
        \|\grad w \tensor v\| _{L ^2} 
        + \|\varpi \cross w\| _{L ^2} 
        + \|\varpi \cross u + \omega \cross w \| _{L ^\frac65}
    \right).
    \label{IW}
\end{align}

In summary, we conclude that for $-4 \le t \le 0$,
\begin{align}
    \label{energy-inequality}
    \frac{\d}{\d t} \int \vp ^4 \frac{|v| ^2}2 \d x 
    + \int \vp ^4 |\grad v| ^2 \d x
    &\le I _\La + I _\RR + I _\BB + I _\BL + I _\BW
\end{align}
with good estimates on each of the term on the right.

\subsection{Proof of Proposition \ref{prop:v-L-inf-L-2}}
First we check the integrability of each terms.
\begin{lemma}[Integrability]
\label{lem:integrability}
Given conditions \eqref{eqn:mean-zero-velocity} and \eqref{eqn:compensated-integrability},
we have
\begin{align}
    \notag
    \| u \| _{\Lp1 \Lq3 (Q _2)} &\le \eta, \\
    \notag
    \| \vp \omega \| _{\Lp1 \Lq1 ((-4, 0) \times \R ^3)} &\le \eta, \qquad
    \| \vp \omega \| _{\Lp2 \Lq2 ((-4, 0) \times \R ^3)} \le \eta, \\
    \notag
    \|\grad v\| _{\Lp1 \Lq1 ((-4, 0) \times \R ^3)} &\le \eta, \qquad  
    \|\grad v\| _{\Lp2 \Lq2 ((-4, 0) \times \R ^3)} \le \eta, \\
    \notag
    \| v \| _{\Lp1 \Lq3 ((-4, 0) \times \R ^3)} &\le \eta, \qquad
    \| v \| _{\Lp2 \Lq4 ((-4, 0) \times \R ^3)} \le \eta, \\
    \notag
    \| \grad w \| _{\Lp1 \Lq3 ((-4, 0) \times \R ^3)} &\le \eta, \\
    \label{eqn:integrability-w}
    \| w \| _{\Lp1 \Lq5 ((-4, 0) \times \R ^3)} &\le \eta, \\
    \label{eqn:integrability-varpi}
    \| \varpi \| _{\Lp1 \Lq3 ((-4, 0) \times \R ^3)} &\le \eta, \qquad
    \| \varpi \| _{\Lp2 \Lq4 ((-4, 0) \times \R ^3)} \le \eta.
\end{align}
\end{lemma}

\begin{proof}
Integrability of $u$ is obtained by Sobolev embedding and that $\vp u$ has average $0$. Integrability of $\vp \omega$ is given. The remaining are consequences of Lemma \ref{lem:integrability-1} and Sobolev embedding.
\end{proof}

\begin{proof}[Proof of Proposition \ref{prop:v-L-inf-L-2}]
We prove Proposition \ref{prop:v-L-inf-L-2} using a Gr\"onwall argument. Multiply \eqref{energy-inequality} by an increasing smooth function $\psi _1 (t)$ with $\psi _1 (t) = 0$ for $t \le -2$, $\psi _1 (t) = 1$ for $t \ge 0$, we have
\begin{align*}
    &\frac{\d}{\d t} \left(
        \psi _1 (t) \int \vp ^4 \frac{|v| ^2}2 \d x
    \right) + \psi _1 (t) \int \vp ^4 |\grad v| ^2 \d x \\
    &\qquad = 
    \psi' _1 (t) \int \vp ^4 \frac{|v| ^2}2 \d x + \psi _1 (t) \left(
        I _\La + I _\RR + I _\BB + I _\BL + I _\BW
    \right).
\end{align*}
Integrate from $-4$ to $t < 0$ we have
\begin{align*}
    & \psi _1 (t) \int \vp ^4 \frac{|v| ^2}2 \d x
    + \int _{-2} ^t \psi _1 (s) \int \vp ^4 |\grad v| ^2 \d x \\
    & \qquad = 
    \int _{-2} ^t \psi' _1 (s) \int \vp ^4 \frac{|v| ^2}2 \d x \d t 
    + \int _{-2} ^t \psi _1 (s) \left(
        I _{\Delta,\RR,\BB,\BL,\BW}
    \right) \d t.
\end{align*}
Because \eqref{Ila}, \eqref{IR}, \eqref{IB}, \eqref{IL}, \eqref{IW}, 
and
\begin{align*}
    \| \vp ^2 v \| _{L ^2 (B _2)}, 
    \| \vp ^2 v \| _{L ^2 (B _2)} ^\frac13 \le C \left(
        1 + \int \vp ^4 |v| ^2 \d x
    \right),
\end{align*}
we can conclude that
\begin{align*}
    &\frac{\d}{\d t} \left(
        \psi _1 (t) \int \vp ^4 \frac{|v| ^2}2 \d x
    \right) + \psi _1 (t) \int \vp ^4 |\grad v| ^2 \d x \\
    &\qquad \le C \Phi (t) \left(
        1 + \psi _1 (t) \int \vp ^4 \frac{|v| ^2}2 \d x
    \right),
\end{align*}
where
\begin{align*}
    \Phi (t) &= \psi' _1 (t) \int \vp ^4 \frac{|v| ^2}2 \d x \\
    &\qquad 
        + \| v \| _{L ^2} 
        + \|u \tensor v\| _{L ^2} 
        + \| \omega \cross u \| _{L ^\frac65 (B _2)} 
    \\
    &\qquad 
        + \| v \| _{\Lq3} ^\frac23 \| \omega \| _{\Lq2 (B _2)}
        + \|\grad w \tensor v\| _{L ^2} 
    \\
    &\qquad 
        + \|\varpi \cross w\| _{L ^2} 
        + \|\varpi \cross u + \omega \cross w \| _{L ^\frac65}
    \\
    &\le
        \| v \| _{L ^2} ^2 
        + \| v \| _{L ^2} 
        + \| v \| _{\Lq4} 
            \| u \| _{\Lq3 (B _2)} 
        + \|\omega\| _{\Lq2 (B _2)} 
            \|u\| _{\Lq3 (B _2)} \\
    &\qquad
        + \| v \| _{\Lq3} ^\frac23 
            \| \omega \| _{\Lq2 (B _2)}
        + \| \grad w \| _{\Lq3} 
            \| v \| _{\Lq4}
    \\
    &\qquad 
        + \| \varpi \| _{\Lq4} 
            \| w \| _{\Lq3}
    \\
    &\qquad 
        + \|\varpi\| _{\Lq4} 
            \|u\| _{\Lq3} 
        + \|\omega\| _{\Lq2} 
            \|w\| _{\Lq3} 
    \\
    &\le \left(
        \| v \| _{\Lq3} 
        + \| v \| _{\Lq3} ^\frac12
        + \| u \| _{\Lq3 (B _2)} 
        + \| \grad w \| _{\Lq3} 
        + \|w\| _{\Lq3} 
    \right)
    \\
    &\qquad \cross 
    \left(
        \| v \| _{\Lq4} 
        + \| v \| _{\Lq4} ^\frac12
        + \|\omega\| _{\Lq2} 
        + \|\varpi\| _{\Lq4} 
    \right)
\end{align*}
Here we used interpolation for
$
    \| v \| _{L ^2} ^2 \le \| v \| _{\Lq3} \| v \| _{\Lq4}. 
$
Therefore
\begin{align*}
    \| \Phi \| _{L ^1 _t} 
    & \lesssim
    \left\|\left(
        \| v \| _{\Lq3} 
        + \| v \| _{\Lq3} ^\frac12
        + \| u \| _{\Lq3 (B _2)} 
        + \| \grad w \| _{\Lq3} 
        + \|w\| _{\Lq3} 
    \right)\right\| _{\Lp1}
    \\
    & \qquad \times
    \left\|\left(
        \| v \| _{\Lq4} 
        + \| v \| _{\Lq4} ^\frac12
        + \|\omega\| _{\Lq2} 
        + \|\varpi\| _{\Lq4} 
    \right)\right\| _{\Lp2}
    \le \eta.
\end{align*}
By a Gr\"onwall's lemma, we conclude that for every $-4 \le t \le 0$,
\begin{align*}
    1 + \psi _1 (t) \int \vp ^4 \frac{|v| ^2}{2} \d x + \int _{-4} ^t \psi _1 (t) \int \vp ^4 |\grad v| ^2 \d x \le e^{\int _{-4} ^t C \Phi (s) \d s} \le e ^{C\eta}.
\end{align*}
Therefore by taking the sup over $-1 \le t \le 0$ and $t = 0$ respectively, we conclude
\begin{align*}
    \sup _{-1 \le t \le 0} \int |v (t)| ^2 \d x \le \eta, \qquad \int _{Q _1} |\grad v| ^2 \d x \d t \le \eta.
\end{align*}

\end{proof}

\section{Local Study: Part Two, De Giorgi Iteration}

\newcommand{\ak}{\alpha _k}
\newcommand{\bk}{\beta _k}
\newcommand{\ek}{\rho _k}
\newcommand{\ekp}{\ek ^\sharp}
\newcommand{\rks}{r _k ^\sharp}
\newcommand{\rkn}{r _k ^\natural}
\newcommand{\rkf}{r _k ^\flat}
\newcommand{\Tks}{T _k ^\sharp}
\newcommand{\Tkn}{T _k ^\natural}
\newcommand{\Tkf}{T _k ^\flat}
\newcommand{\Bks}{B _k ^\sharp}
\newcommand{\Bkn}{B _k ^\natural}
\newcommand{\Bkf}{B _k ^\flat}
\newcommand{\Qks}{Q _k ^\sharp}
\newcommand{\Qkn}{Q _k ^\natural}
\newcommand{\Qkf}{Q _k ^\flat}

In this section, we derive the boundedness of $v$ in $Q _\frac12$ which is the following.
\begin{proposition}
\label{prop:v-L-inf}
Let $v$ solves \eqref{eqn-v-final}. If \eqref{eqn:v-L-infty-L-2} holds for sufficiently small $\eta$, and we have integrability bounds in Lemma \ref{lem:integrability}, then we have
\begin{align*}
    \| v \| _{L ^\infty (Q _\frac12)} = \sup _{t \in (-1, 0)} \| v (t) \| _{L ^\infty (B _\frac12)} \le 1.
\end{align*}
\end{proposition}

The proof uses De Giorgi technique and the truncation method. First, we set dyadically shrinking radius,
\begin{align*}
    \rkf &= \frac12(1 + 8 ^{-k}),  & 
    \rkn &= \frac12(1 + 2 \times 8 ^{-k}), &
    \rks &= \frac12(1 + 4 \times 8 ^{-k}).
\end{align*}
Then we define dyadically shrinking cylinder $Q _k$'s,
\begin{align*}
    \Tkf &= {\rkf} ^2,  &
    \Bkf &= B _{\rkf} (0), &
    \Qkf &= (-\Tkf, 0) \times \Bkf, \\
    \Tkn &= {\rkn} ^2,  &
    \Bkn &= B _{\rkn} (0), &
    \Qkn &= (-\Tkn, 0) \times \Bkn, \\
    \Tks &= {\rks} ^2,  &
    \Bks &= B _{\rks} (0), &
    \Qks &= (-\Tks, 0) \times \Bks.
\end{align*}
We also introduce positive smooth space-time cut-off functions 
$\ek$ and $\ekp$ with 
\begin{align*}
    \ind{\Qkf} \le \ek \le \ind{\Qkn}, \qquad \ind{\Qks} \le \ekp \le \ind{Q ^\flat _{k - 1}}.
\end{align*}
Then, let $c _k$ denote a sequence of rising energy level, 
\begin{align*}
    c _k &= 1 - 2 ^{-k}, & 
    v _k &= (|v| - c _k) _+, & 
    \bk &= \frac{v _k}{|v|}, \\
    \Omega _k &= \{ v _k > 0 \}, &
    \ind k &= \ind{\Omega _k}, &
    \ak &= 1 - \bk.
\end{align*}
We define analogous of vector derivative $d _k$ and energy quantity $U _k$:
\begin{align*}
    d _k ^2 &= \ind k \left(
        \ak |\grad |v||^2 + \bk |\grad v| ^2
    \right), \\
    U _k &= \| v _k \| _{L ^\infty (-\Tkf, 0; L ^2 (\Bkf))} ^2 + \| d _k \| _{L ^2 (\Qkf)} ^2.
\end{align*}

We have the following truncation estimates.
\begin{lemma}
    \begin{align*}
        \ak v \le c _k &\le 1, \\
        \|\bk v\| _{\LLLH (Q _{k-1}^\flat)} ^2 &\le 9 U _{k - 1}, \\
        \|\ind k \| _{L ^\infty _t L ^2 _x \cap L ^2 _t L ^6 _x (Q _{k-1} ^\flat)} ^2 &\le C ^k U _{k - 1}.
    \end{align*}
\end{lemma}

\begin{proof}
    The first estimate follows from the definition. By Lemma 4 in \cite{Vasseur2007}, we have $|\grad v _k| \le d _k$ and $|\grad (\bk v)| \le 3 d _k$. Moreover, since $|\grad|v|| \le |\grad v| ^2$, we see $d _k \le d _{k - 1}$, as $v _k$ and $\bk$ are monotonously decreasing. So
    \begin{align*}
        \| \grad (\bk v) \| _{L ^2 (Q _{k - 1} ^\flat)} \le 
        3\| d _k \| _{L ^2 (Q _{k - 1} ^\flat)} \le 
        3\| d _{k - 1} \| _{L ^2 (Q _{k - 1} ^\flat)}.
    \end{align*}
    Moreover, the truncation gives $|\bk v| + 2 ^{-k} \ind k = v _k + 2 ^{-k} \ind k = \ind k v _{k - 1}$, so
    \begin{align*}
        \| \bk v \| _{L ^\infty _t L ^2 _x (Q _{k - 1} ^\flat)} 
        &\le \|v _{k - 1}\| _{L ^\infty _t L ^2 _x (Q _{k - 1} ^\flat)}, \\
        2 ^{-k} \| \ind k \| _{L ^\infty _t L ^2 _x (Q _{k - 1} ^\flat)} &\le \|v _{k - 1}\| _{L ^\infty _t L ^2 _x (Q _{k - 1} ^\flat)} , \\
        2 ^{-k} \| \ind k \| _{L ^2 _t L ^6 _x (Q _{k - 1} ^\flat)} 
        &\le \| v _{k - 1} \| _{L ^2 _t L ^6 _x (Q _{k - 1} ^\flat)} \\
        &\le \| v _{k - 1} \| _{L ^\infty _t L ^2 _x (Q _{k - 1} ^\flat)} 
        + \| \grad v _{k - 1} \| _{L ^2 (Q ^\flat _{k - 1})} \\
        &\le \| v _{k - 1} \| _{L ^\infty _t L ^2 _x (Q _{k - 1} ^\flat)} 
        + \| d _{k - 1} \| _{L ^2 (Q ^\flat _{k - 1})}.
    \end{align*}
\end{proof}

\begin{corollary}[Nonlinearization]
\label{cor:nonlinearize}
If $f \in L ^p _t L ^q _x (Q _{k - 1})$, with
\begin{align*}
    \frac1p + \gamma \left(
        \frac\theta2 + \frac{1-\theta}\infty
    \right) = 1, \qquad
    \frac1q + \gamma \left(
        \frac\theta6 + \frac{1-\theta}2
    \right) = 1,
\end{align*}
for some $0 \le \theta \le 1$, $0 < \sigma \le \gamma$, then uniformly in $\sigma$,
\begin{align*}
    \int _{Q _{k - 1} ^\flat} |\bk v| ^\sigma |f| \d x \d t \le C ^k \| f \| _{L ^p _t L ^q _x (Q _{k - 1})} U _{k - 1} ^\frac\gamma2.
\end{align*}
\end{corollary}

\newcommand{\Lpt}{L ^{p _\theta} _t}
\newcommand{\Lqt}{L ^{q _\theta} _x}

\begin{proof}
By interpolation, 
\begin{align*}
    \| \bk v \|, \| \ind k \| _{\Lpt \Lqt (Q _{k - 1})} \le U _{k - 1} ^\frac12,
\end{align*}
where
\begin{align*}
    \frac1{p _\theta} = \frac\theta2 + \frac{1-\theta}\infty
    , \qquad
    \frac1{q _\theta} = 
        \frac\theta6 + \frac{1-\theta}2.
\end{align*}
Therefore, using H\"older's inequality,
\begin{align*}
    \int _{Q _{k - 1}} |\bk v| ^\sigma |f| \d x \d t 
    \le 
    \| f \| _{L ^p _t L ^q _x} 
    \| \bk v \| _{\Lpt \Lqt} ^\sigma 
    \| \ind k \| _{\Lpt \Lqt} ^{\gamma - \sigma} 
    \le
    \| f \| _{L ^p _t L ^q _x} U _{k - 1} ^\frac\gamma2.
\end{align*}
\end{proof}

First, we recall the following identities from \cite{Vasseur2007}.
\begin{align}
    \label{ak-algebra}
    \ak v \cdot \partial _\bullet v &= \partial _{\bullet} \left(\frac{|v| ^2 - v _k ^2}{2}\right), \\
    \label{ak-3}
    \ak v \cdot \La v &= \La \left(
        \frac{|v|^2 - v _k ^2}2
    \right) + d _k ^2 - |\grad v| ^2.
\end{align}
Since $\ak v$ is bounded, we can multiply equation \eqref{eqn-v-final} by $\ak v$ and obtain
\begin{align}
    \label{ak-2}
    &\pt \left(
        \frac{|v|^2 - v _k ^2}2
    \right) 
    + \ak v \cdot \grad \RR (u \tensor v)
    \\
    \notag
    &\qquad
    = \La \left(
        \frac{|v|^2 - v _k ^2}2
    \right) + d _k ^2 - |\grad v| ^2
    + \ak v \cdot (\BB + \BL + \BW).
\end{align}
using \eqref{ak-algebra} and \eqref{ak-3}. Denote $\Cv = \BB + \BL + \BW$. Subtracting \eqref{ak-2} from \eqref{local-inequality}, we have
\begin{align*}
    \pt \frac{v _k ^2}{2} + d _k ^2 
    + \div (v \RR (u \tensor v)) - \ak v \cdot \grad \RR (u \tensor v)
    \le \La \frac{v _k ^2}{2} 
    + \bk v \cdot \Cv.
\end{align*}
Multiply by $\ek$, then integrate in space and from $\sigma$ to $\tau$ in time,
\begin{align*}
    &\left[
        \int \ek \frac{v _k ^2}{2} \d x
    \right] ^\tau _\sigma 
    + \int _\sigma ^\tau \int \ek d _k ^2 \d x \d t \\
    &\qquad \le \int _\sigma ^\tau \int (\pt \ek + \La \ek) \frac{v _k ^2}{2} \d x \d t 
    - \int _\sigma ^\tau \int \ek \div (v \RR (u \tensor v)) \d x \d t \\
    &\qquad \qquad + \int _\sigma ^\tau \int \ek \ak v \cdot \grad \RR (u \tensor v) \d x \d t 
    + \int _\sigma ^\tau \int \ek \bk v \cdot \Cv \d x \d t.
\end{align*}
Take the sup over $\tau > -\Tkf$, and set $\sigma < -T _{k - 1} ^\flat$, we obtain
\begin{align}
    \label{nonlinearize}
    U _k &\le
    \sup _{\tau \in (-\Tkf, 0)} \int \ek \frac{v _k ^2}{2} \d x + \int _{-T _{k - 1} ^\flat} ^0 \int \ek d _k ^2 \d x \d t \\
    \notag
    & \le C ^k \int _{\Qkn} v _k ^2 \d x \d t 
    + \sup _{\tau \in (-\Tkf, 0)} 
    \bigg\lbrace
    \int _{-\Tkn} ^\tau \int 
    \ek \ak v \cdot \grad \RR (u \tensor v) 
    \d x \d t \\
    \notag
    & \qquad \qquad \qquad \qquad \qquad \qquad \qquad -
    \int _{-\Tkn} ^\tau \int 
    \ek \div (v \RR (u \tensor v))
    \d x \d t \\
    \notag
    & \qquad \qquad \qquad \qquad \qquad \qquad \qquad +
    \int _{-\Tkn} ^\tau \int 
    \ek \bk v \cdot \Cv \d x \d t \bigg\rbrace.
\end{align}
Using Corollary \ref{cor:nonlinearize}, the first one is bounded by 
\begin{align}
    \label{eqn:vk2-term}
    \int _{\Qkn} v _k ^2 \d x \d s 
    \le 
    \int _{Q _{k - 1} ^\flat} |\bk v|^2 \d x \d s
    \le U _{k - 1} ^\frac53.
\end{align}
Now let's deal with the last few terms. For simplicity, we use $\iint \d x \d t$ to denote $\int _{-\Tkn} ^\tau \int _{\R ^3} \d x \d t$ in the rest of this section.

\subsection{Highest Order Nonlinear Term}
Define three trilinear forms,
\newcommand{\Tc}{\mathbf T _\circ}
\newcommand{\Tg}{\mathbf T _\grad}
\newcommand{\Td}{\mathbf T _{\div}}
\begin{align*}
    \Tc [v _1, v _2, v _3] &= \iint \ek \div (v _1 \RR (v _2 \tensor v _3)) \d x \d t, \\
    \Tg [v _1, v _2, v _3] &= \iint \ek v _1 \cdot \grad \RR (v _2 \tensor v _3) \d x \d t, \\
    \Td [v _1, v _2, v _3] &= \iint \ek \div v _1 \RR (v _2 \tensor v _3) \d x \d t.
\end{align*}
They are symmetric on $v _2$, $v _3$ positions. When we have enough integrability, that is, when
\begin{align*}
    |\grad v _1| |v _2| |v _3|, |v _1| |\grad v _2| |v _3|, |v _1| |v _2| |\grad v _3| \in L ^1 _{t, x},
\end{align*}
we have Leibniz rule
\begin{align*}
    \Tc = \Tg + \Td.
\end{align*}
The goal is to estimate the first two double integrals in \eqref{nonlinearize},
\begin{align*}
    &\iint 
    \ek \ak v \cdot \grad \RR (u \tensor v) \d x \d t - \iint \ek \div (v \RR (u \tensor v)) \d x \d t \\
    &\qquad
    = \Tg [\ak v, u, v] - \Tc [v, u, v].
\end{align*}

We first separate $w \tensor v$ from $u \tensor v$, and we will have
\begin{align*}
    \Tg [\ak v, w, v] - \Tc [v, w, v] &= \Tg [\ak v, w, v] - \Tg [v, w, v] - \Td [v, w, v]\\
    &= - \Tg [\bk v, w, v]
    \\
    &= 
    -\iint \ek \bk v \cdot \grad \RR (w \tensor v) \d x \d t.
\end{align*}
Denote $-\grad \RR (w \tensor v) =: \BW _2$ and we will deal with it later. The remaining $(u - w) \tensor v$ can be separated into interior part and exterior part,
\begin{align*}
    (u - w) \tensor v = \ekp v \tensor v + (1 - \ekp) (u - w) \tensor v.
\end{align*}
The exterior part is bounded and smooth in space over the support of $\ek$. 
\begin{align*}
    \| \ek \RR ((1 - \ekp) (u - w) \tensor v) \| _{\Lp3 C ^\infty _x} 
    &\le C \| (u - w) \tensor v \| _{\Lp3 L ^2 _x} 
    \\
    &\le C \| u - w \| _{\Lp1 \Lq3 (Q _2)} \| v \| _{\Lp2 \Lq4}
    \le \eta.
\end{align*}
Here, we denote
\begin{align*}
    \frac1{p _3} = \frac1{p _1} + \frac1{p _2} < 1.
\end{align*}

Therefore we can use Leibniz rule similar as $w$ and
\begin{align*}
    &\Tg [\ak v, (1 - \ekp) (u - w), v] - \Tc [v, (1 - \ekp) (u - w), v] \\
    &= \Tg [\ak v, (1 - \ekp) (u - w), v] - \Tg [v, (1 - \ekp) (u - w), v] \\
    &= -\Tg [\bk v, (1 - \ekp) (u - w), v] \\
    &= -\iint \ek \bk v \cdot \grad \RR (v \tensor (1 - \ekp) (u - w)) \d x \d t \\
    &
    \le C ^k U _{k - 1} ^{\frac53 - \frac2{3 p _3}}
\end{align*}
by nonlinearization Corollary \ref{cor:nonlinearize}.
The interior part is
\begin{align*}
    &
    \Tg [\ak v, \ekp v, v] 
    - \Tc [v, \ekp v, v] 
    \\
    &= 
    \Tg [\ak v, \ekp \bk v, \bk v] 
    + 2\Tg [\ak v, \ekp \ak v, \bk v] 
    \\
    &\qquad 
    + \Tg [\ak v, \ekp \ak v, \ak v] 
    - \Tc [v, \ekp v, v] 
    \\
    &= 
    \Tg [\ak v, \ekp \bk v, \bk v] 
    \\
    &\qquad 
    + 2\Tc [\ak v, \ekp \ak v, \bk v] 
    - 2\Td [\ak v, \ekp \ak v, \bk v] 
    \\
    &\qquad 
    + \Tc [\ak v, \ekp \ak v, \ak v] 
    - \Td [\ak v, \ekp \ak v, \ak v] 
    \\
    &\qquad 
    - \Tc [v, \ekp v, v] 
    \\
    &= 
    \Tg [\ak v, \ekp \bk v, \bk v] 
    \\
    &\qquad 
    + 2\Td [\bk v, \ekp \ak v, \bk v] 
    + \Td [\bk v, \ekp \ak v, \ak v] 
    \\
    &\qquad 
    + 2\Tc [\ak v, \ekp \ak v, \bk v] 
    + \Tc [\ak v, \ekp \ak v, \ak v] 
    \\
    &\qquad 
    - \Tc [v, \ekp v, v] 
    \\
    &= 
    \Tg [\ak v, \ekp \bk v, \bk v] 
    + \Td [\bk v, \ekp \ak v, (\bk + 1) v] 
    \\
    &\qquad 
    - \Tc [\ak v, \ekp \bk v, \bk v] 
    - \Tc [\bk v, \ekp v, v].
\end{align*}
Notice that the boundedness of $\ak v$ guarantees enough integrability to switch between trilinear forms. Then
\begin{align*}
    &|\Tg [\ak v, \ekp \bk v, \bk v]|, 
    |\Td [\bk v, \ekp \ak v, (\bk + 1) v]| \\
    &\qquad \lesssim \| \grad (\bk v) \| _{L ^2 (Q _{k - 1})} U _{k - 1} ^\frac56 \le U _{k - 1} ^\frac43, \\
    & |\Tc [\ak v, \ekp \bk v, \bk v]|,
    |\Tc [\bk v, \ekp v, v]| 
    \lesssim U _{k - 1} ^\frac53.
\end{align*}
In conclusion, 
\begin{align}
    \label{eqn:trilinear}
    &\bigg|
        \iint \ek \ak v \cdot \grad \RR (u \tensor v) \d x \d t 
        - \iint \ek \div (v \RR (u \tensor v)) \d x \d t
        \\
    \notag
        &\qquad
        - \iint \ek \bk v \cdot \BW _2 \d x \d t
    \bigg| \lesssim C ^k U _{k - 1} ^{\min\{\frac43, \frac53-\frac23{p_3}\}}.
\end{align}

\subsection{Lower Order Terms}
For the bilinear and linear term, recall that inside $B _1$,
\begin{align*}
    \BB &= -\curl \La \inv (\grad \vp \cross (\omega \cross u)), \\
    \BL &= \curl \La \inv \left(2 \div (\grad \vp \tensor \omega) - (\La \vp) \omega \right). 
\end{align*}
Therefore, 
\begin{align*}
    &\| \ek \BB \| _{\Lp3 L ^\infty _x} \le \| \omega \cross u \| _{\Lp3 L ^\frac65 _x (Q _2)} \le \| u \| _{\Lp1 \Lq3} \| \omega \| _{\Lp2 \Lq2} \le \eta, \\
    &\| \ek \BL \| _{\Lp2 L ^\infty _x} \le \| \omega \| _{\Lp2 \Lq2 (Q _2)} \le \eta, 
\end{align*}
Thus
\begin{align}
    \label{eqn:lower-order-1}
    \iint \BB \cdot \ek \bk v \d x \d t
    &\le C ^k U _{k - 1} ^{\frac53 - \frac2{3 p _3}}, \\
    \label{eqn:lower-order-2}
    \iint \BL \cdot \ek \bk v \d x \d t
    &\le C ^k U _{k - 1} ^{\frac53 - \frac2{3 p _2}}.
\end{align}

\subsection{W Terms}
Finally, let us deal with
\begin{align*}
    \BW + \BW _2 &= -\omega \cross w + \frac12 \Pgrad \left(
        \varpi \cross u + \omega \cross w 
    \right) - \grad \RR (w \tensor v).
\end{align*}
Here $\grad \RR = \frac12 \grad \tr - \Pgrad \div$, so
\begin{align*}
    \grad \RR (w \tensor v) 
    &= \frac12 \grad (w \cdot v) - \Pgrad \div (v \tensor w) \\
    &= \frac12 \left(
        w \cdot \grad v + v \cdot \grad w + w \cross \curl v + v \cross \curl w
    \right) - \Pgrad (v \cdot \grad w) \\
    &= \frac12 \left(
        w \cdot \grad v - v \cdot \grad w
    \right) + \Pcurl (v \cdot \grad w) \\
    & \qquad  + \frac12 \left(
        w \cross \curl v + v \cross \curl w
    \right), \\
    \grad \RR (w \tensor v) &= \Pgrad (\grad \RR (w \tensor v)) \\
    &= \frac12 \Pgrad \left(
        w \cdot \grad v - v \cdot \grad w
    \right) + \frac12 \Pgrad \left(
        w \cross \curl v + v \cross \curl w
    \right) \\
    &= \frac12 \Pgrad \left(
        \curl (v \cross w) - v \div w + w \div v
    \right) \\
    &\qquad 
    + \frac12 \Pgrad \left(
        w \cross \curl v + v \cross \curl w
    \right) \\ 
    &= -\frac12 \Pgrad \left(
        v (u \cdot \grad \vp)
    \right) + \frac12 \Pgrad \left(
        w \cross \curl v + v \cross \curl w
    \right). 
\end{align*}
Hence
\begin{align*}
    \BW + \BW _2 &= -\omega \cross w + \frac12 \Pgrad \left(
        v (u \cdot \grad \vp)
    \right) \\
    &\qquad + \frac12 \Pgrad \left(
        \varpi \cross u + \omega \cross w 
        + \curl v \cross w + \curl w \cross v
    \right).
\end{align*}
Again, we separate $\BW + \BW _2$ into exterior and interior part, with
\begin{align*}
    \BW + \BW _2 = \BW _{\mathrm{ext}} + \BW _{\mathrm{int}}
\end{align*}
where
\begin{align*}
    \BW _{\mathrm{ext}} &= -(1 - \ekp) \omega \cross w + \frac12 \Pgrad \left(
        v (u \cdot \grad \vp)
    \right) \\
    &\qquad + \frac12 \Pgrad \left(
        (1 - \ekp) \left(
            \varpi \cross u + \omega \cross w 
            + \curl v \cross w + \curl w \cross v
        \right)
    \right), \\
    \BW _{\mathrm{int}} &= -\ekp \omega \cross w \\
    &\qquad + \frac12 \Pgrad \left(
        \ekp \left(
            \varpi \cross u + \omega \cross w 
            + \curl v \cross w + \curl w \cross v
        \right)
    \right) \\
    &= -\ekp \curl v \cross w - \ekp \varpi \cross w \\
    &\qquad + \frac12 \Pgrad \left(
        \ekp \left(
            \varpi \cross u + \curl w \cross v + \varpi \cross w 
        \right)
    \right)\\
    &\qquad + \frac12 \Pgrad \left(
        \ekp \left(
            \omega \cross w + \curl v \cross w - \varpi \cross w
        \right)
    \right) \\
    &= -\ekp \curl v \cross w - \ekp \varpi \cross w \\
    &\qquad + \Pgrad \left(
        \ekp \varpi \cross u 
    \right) + \Pgrad \left(
        \ekp \curl v \cross w
    \right) \\
    &= -\Pcurl (\ekp \curl v \cross w) - \Pcurl(\ekp \varpi \cross w) + \Pgrad \left(
        \ekp \varpi \cross v
    \right).
\end{align*}
Similar as bilinear terms, $\ek \BW _{\mathrm{ext}}$ is small in $\Lp3 L ^\infty _x$. Among the three terms in $\BW _{\mathrm{int}}$, $\ekp \varpi \cross w$ is bounded in $\Lp3 L ^\infty _x$, and $\ekp \varpi$ is in $\Lp2 L ^\infty _x$. Finally, for the first term,
\begin{align*}
    \Pcurl (\curl v \cross \ekp w) &= -\Pcurl (\curl \ekp w \cross v) + \Pcurl(v \cdot \grad \ekp w + \ekp w \cdot \grad v), \\
    \Pcurl(\ekp w \cdot \grad v) &= \Pcurl(
        \curl (v \cross \ekp w) + v \cdot \grad \ekp w - v \div \ekp w
    ) \\
    &= \curl (v \cross \ekp w)
    + \Pcurl(
        v \cdot \grad \ekp w - v \div \ekp w
    ), \\
    \curl (v \cross \ekp w) &= v \div \ekp w + \ekp w \cdot \grad v - v \cdot \grad \ekp w.
\end{align*}
Every term is a product of $v$ and $\grad \ekp w$ (possibly with a Riesz transform) except $\ekp w \cdot \grad v$.
Because in $\Omega _k$, $\grad |v| = \grad v _k$ are the same, we have
\begin{align*}
    \int \ek \bk v \cdot (\ekp w \cdot \grad) v \d x &=
    \int \ek \bk (w \cdot \grad) \frac{|v| ^2}{2} \d x \\
    &= \int \ek \bk |v| (w \cdot \grad) |v| \d x \\
    &= \int \ek v _k (w \cdot \grad) v _k \d x \\
    &= \int \ek (w \cdot \grad) \frac{v _k ^2}2 \d x \\
    &= -\int \frac{v _k ^2}{2} \div (\ek w) \d x.
\end{align*}
Therefore, every term of $\Pcurl (\curl v \cross \ekp w)$ is a product of $v$ and $\grad \ek w$ or $\grad \ekp w$. Inside $B _1$, $w \in \Lp1 C ^\infty _x$. In conclusion,
\begin{align*}
    \iint \ek \bk v \cdot \BW _{\textrm{ext}} \d x \d t &\le C ^k U _{k - 1} ^{\frac53 - \frac2{3 p _3}}, \\
    \iint \ek \bk v \cdot \Pcurl (\ekp \curl v \cross w) \d x \d t &\le C ^k U _{k - 1} ^{\frac53 - \frac2{3 p _1}}, \\
    \iint \ek \bk v \cdot \Pcurl(\ekp \varpi \cross w) \d x \d t &\le C ^k U _{k - 1} ^{\frac53 - \frac2{3 p _3}}, \\
    \iint \ek \bk v \cdot \Pgrad (\ekp \varpi \cross v) \d x \d t &\le C ^k U _{k - 1} ^{\frac53 - \frac2{3 p _2}}.
\end{align*}
So the sum is bounded in
\begin{align}
    \label{eqn:W+W2}
    \iint \ek \bk v \cdot (\BW + \BW _2) \d x \d t &= 
    \iint \ek \bk v \cdot (\BW _{\textrm{int}} + \BW _{\textrm{ext}}) \d x \d t \le C ^k U _{k - 1} ^{\frac53 - \frac2{3 p _3}}
\end{align}
provided $U _{k - 1} < 1$.

\subsection{Proof of Proposition \ref{prop:v-L-inf}}

\begin{proof}[Proof of Proposition \ref{prop:v-L-inf}]
Coming back to \eqref{nonlinearize}, by estimates \eqref{eqn:vk2-term} on the first term, \eqref{eqn:trilinear} on the trilinear terms, \eqref{eqn:lower-order-1}, \eqref{eqn:lower-order-2} on the $\BB, \BL$ terms and \eqref{eqn:W+W2} on the $\BW$ terms, we conclude that
\begin{align*}
    U _{k} \le C ^k U _{k - 1} ^{\min\{\frac53 - \frac2{3 p _3}, \frac43\}}
\end{align*}
provided $U _{k - 1} < 1$. Here $p _3 > 1$ ensures the index is strictly greater than 1. Since
\begin{align*}
    U _0 &= \sup _{t \in (-1, 0)} \int |v _0| ^2 \d x + \int _{-1} ^0 \int _{B _1} d _0 ^2 \d x \d t \\
    &= \sup _{t \in (-1, 0)} \int |v| ^2 \d x + \int _{-1} ^0 \int _{B _1} |\grad v| ^2 \d x \d t \le \eta
\end{align*}
by Proposition \ref{prop:v-L-inf-L-2}, we know that if $\eta$ is small enough, $U _k \to 0$ as $k \to \infty$. So in $Q _\frac12$, $|v| \le 1$ a.e.. This finishes the proof of Proposition \ref{prop:v-L-inf}.
\end{proof}

\section{Local Study: Part Three, More Regularity}

In this section, we will show that the vorticity $\omega$ is smooth in space. We will only work with the vorticity equation from now on. After the previous two steps, in $B _\frac12$ we should always decompose $u = v + w$, because $v$ is bounded and $w$ is harmonic. 

For convenience, given a vector $\omega$, we denote
\begin{align*}
    \omega ^0 := \frac{\omega}{|\omega|}, \qquad \omega ^\alpha := |\omega| ^\alpha \omega ^0, \alpha \in \R.
\end{align*}
\newcommand{\pb}{\partial _\bullet}
Let $\pb$ be the partial derivative in any space direction or time, then we have
\begin{align*}
    \pb (|\omega| ^\alpha) 
    &= \alpha \omega ^{\alpha - 1} \cdot \pb \omega, \\
    \pb (\omega ^\alpha) 
    &= |\omega| ^{\alpha - 1} \pb \omega + (\alpha - 1) (\omega ^{\alpha - 2} \cdot \pb \omega)
    \omega, \\
    \frac1\alpha \pb \pb (|\omega| ^\alpha) 
    &= |\omega| ^{\alpha - 2} |\pb \omega| ^2 + (\alpha - 2) 
    (\omega ^{\frac\alpha2 - 1} \cdot \pb \omega) ^2
    + \omega ^{\alpha - 1} \cdot \pb \pb \omega \\ 
    &\ge (\alpha - 1) (\omega ^{\frac\alpha2 - 1} \cdot \pb \omega) ^2
    + \omega ^{\alpha - 1} \cdot \pb \pb \omega \\ 
    &=  \frac{4(\alpha - 1)}{\alpha ^2} \left| \pb \omega ^{\frac\alpha2} \right| ^2
    + \omega ^{\alpha - 1} \cdot \pb \pb \omega.
\end{align*}

\newcommand{\grada}{\grad ^\alpha}
\newcommand{\gradb}{\grad ^\beta}



\subsection{Bound Vorticity in the Energy Space}
We will first show $\omega$ is bounded in the energy space.

\begin{proposition}
\label{prop:omega-energy-space}
If $u = v + w$ in $Q _\frac12$, where $v, w$ are bounded in
\begin{align}
    \label{eqn:hypotheses-for-v-w}
    \| v \| _{L ^\infty (Q _\frac12)}
    + \| \grad v \| _{L ^2 (Q _\frac12)} \le 2, \\
    \label{eqn:hypotheses-for-v-w-2}
    \| \curl w \| _{L ^2 _t L ^\frac32 _x (Q _\frac12)}
    + \| w \| _{L ^\frac43 _t \Lip _x (Q _\frac12)}
    \le 2,
\end{align}
$\omega = \curl u$ solves the vorticity equation \eqref{eqn:vorticity},
then 
\begin{enumerate}[\upshape (a)]
    \item $\| \omega ^\frac34 \| _{\mathcal E (Q _\frac14)} \le C$,
    \item $\| \omega \| _{\mathcal E (Q _\frac18)} \le C$,
\end{enumerate}
\end{proposition}

\begin{proof}[Proof of Proposition \ref{prop:omega-energy-space} (a)]
We fix a pair of smooth space-time cut-off functions $\vr$ and $\vs$ which satisfy
\begin{align*}
    \ind{Q _\frac18} \le \vs \le \ind{Q _\frac14} \le \vr \le \ind{Q _\frac12}.
\end{align*}

Take the dot product of the vorticity equation \eqref{eqn:vorticity} with $\frac32 \omega ^\frac12$:
\begin{align*}
    \frac32 \omega ^\frac12 \cdot \pt \omega &= \pt (|\omega| ^\frac32), \\
    \frac32 \omega ^\frac12 \cdot (u \cdot \grad) \omega &= (u \cdot \grad) (|\omega| ^\frac32), \\
    \frac32 \omega ^\frac12 \cdot \La \omega &\le \La (|\omega| ^\frac32) - \frac43 |\grad \omega ^\frac34| ^2.
\end{align*}
Therefore,
\begin{align*}
    (\pt + u \cdot \grad - \La) (|\omega|^\frac32) + \frac32 \omega \cdot \grad u \cdot \omega^\frac12 + \frac43 |\grad \omega ^\frac34| ^2 \le 0.
\end{align*}
Multiply by $\vr ^6$ then integrate over space,
\begin{align}
    \label{eqn:varrho-1}
    \int \vr ^6 (\pt + u \cdot \grad - \La) (|\omega|^\frac32) \d x + 
    \frac43 \int \vr ^6 |\grad \omega ^\frac34| ^2 \d x 
    &\le -\frac32 \int \vr ^6 \omega \cdot \grad u \cdot \omega^\frac12 \d x.
\end{align}

For the left hand side, we can integrate by part,
\begin{align}
    \label{eqn:varrho-2}
    &\int \vr ^6 (\pt + u \cdot \grad - \La) (|\omega|^\frac32) \d x
    \\
    \notag
    &\qquad
    = \frac{\d}{\d t} \int \vr ^6 |\omega| ^\frac32 \d x
    - \int \left(
        (\pt + u \cdot \grad + \La) \vr ^6
    \right) |\omega|^\frac32 \d x,
\end{align}
where the latter can be controlled by
\begin{align}
    \label{eqn:varrho-3}
    \int \left(
        (\pt + u \cdot \grad + \La) \vr ^6
    \right) |\omega|^\frac32 \d x
    \le C \left(
        1 + \| u \| _{L ^\infty (B _\frac12)}
    \right) \int \vr ^4 |\omega| ^\frac32 \d x.
\end{align}
For the right hand side, using $u = v + w$ over the support of $\vr$ we can separate
\begin{align}
    \label{eqn:varrho-4}
    \int \vr ^6 \omega \cdot \grad u \cdot \omega ^\frac12 \d x
    = 
    \int \vr ^6 \omega \cdot \grad v \cdot \omega ^\frac12 \d x + 
    \int \vr ^6 \omega \cdot \grad w \cdot \omega ^\frac12 \d x, 
\end{align}
The $\grad v$ term can be controlled by
\begin{align}
    \label{eqn:varrho-5}
    \int \vr ^6 \omega \cdot \grad v \cdot \omega ^\frac12 \d x
    &= -\int \omega \cdot \grad (\vr ^6 \omega ^\frac12) \cdot v \d x
    \\
    \notag
    &= -\int \vr ^6 \omega \cdot \grad (\omega ^\frac12) \cdot v \d x
    - \int \omega \cdot (\omega ^\frac12 \tensor \grad \vr ^6) \cdot v \d x,
\end{align}
where
\begin{align*}
    \omega \cdot \grad (\omega ^\frac12) 
    &= |\omega| ^{-\frac12} \omega \cdot \grad \omega - \frac12 (\omega \cdot \grad \omega \cdot \omega ^{-\frac32})
    \omega 
    = \omega ^\frac12 \cdot \grad \omega - \frac12 (\omega ^\frac12 \cdot \grad \omega \cdot \omega ^0) \omega^0 \\
    \Rightarrow
    |\omega \cdot \grad (\omega ^\frac12)| 
    &\le \left| \frac32 \omega ^\frac12 \cdot \grad \omega \right| 
    = 2 |\omega| ^\frac34 \left| \frac34 \omega ^{-\frac14} \cdot \grad \omega \right|
    = 2 |\omega| ^\frac34 \left| \grad |\omega| ^\frac34 \right|
    \\
    &\le 2 |\omega| ^\frac34 |\grad \omega ^\frac34| 
    \le |\omega| ^\frac32 + |\grad \omega ^\frac43| ^2.
\end{align*}
Here the second to the last inequality is due to $\partial _i |\omega| ^\frac34 = \partial _i \omega ^\frac34 \cdot \omega ^0$. Since $|v| \le 1$ over the support of $\vr$,
\begin{align}
    \label{eqn:varrho-6}
    \int \vr ^6 \omega \cdot \grad (\omega ^\frac12) \cdot v \d x
    &\le 
    \int \vr ^6 |\omega| ^\frac32 \d x
    + \int \vr ^6 | \grad \omega ^\frac34 | ^2 \d x.
\end{align}
By using \eqref{eqn:varrho-2}-\eqref{eqn:varrho-6} in \eqref{eqn:varrho-1}, we conclude
\begin{align*}
    &\frac{\d}{\d t} \int \vr ^6 |\omega| ^\frac32 \d x
    + \frac43 \int \vr ^6 |\grad \omega ^\frac34| ^2 \d x 
    \\
    &\qquad 
    \le \int \left[
        (\pt + u \cdot \grad + \La) \vr ^6
    \right] |\omega|^\frac32 \d x 
    \\
    &\qquad \qquad 
    + \int \vr ^6 \omega \cdot \grad w \cdot \omega ^\frac12 \d x
    \\
    &\qquad \qquad 
    + \int \omega \cdot (\omega ^\frac12 \tensor \grad \vr ^6) \cdot v \d x 
    \\
    &\qquad \qquad 
    +
    \int \vr ^6 |\omega| ^\frac32 \d x
    + \int \vr ^6 | \grad \omega ^\frac34 | ^2 \d x 
    \\
    &
    \frac{\d}{\d t} \int \vr ^6 |\omega| ^\frac32 \d x
    + \frac13 \int \vr ^6 |\grad \omega ^\frac34| ^2 \d x
    \\
    &\qquad 
    \le C \left(
        1
        + \| u (t) \| _{L ^{\infty} (B _\frac12)} 
        + \| \grad w (t) \| _{L ^{\infty} (B _\frac12)}
    \right) \int \vr ^4 |\omega| ^\frac32 \d x.
\end{align*}
By H\"older's inequality,
\begin{align*}
    \int \vr ^4 |\omega| ^\frac32 \d x \le 
    \| \omega (t) \| _{L ^\frac32 (B _\frac12)} ^\frac12 \left(
        \int \vr ^6 |\omega| ^\frac32 \d x
    \right) ^\frac23.
\end{align*}
Therefore we can write
\begin{align*}
    &\frac{\d}{\d t} \int \vr ^6 |\omega| ^\frac32 \d x
    + \frac13 \int \vr ^6 |\grad \omega ^\frac34| ^2 \d x 
    \le C \Phi (t) \left(
        1 +
        \int \vr ^6 |\omega| ^\frac32 \d x
    \right),
\end{align*}
where
\begin{align*}
    \Phi (t) 
    &= \left(
        1
        + \| u (t) \| _{L ^{\infty} (B _\frac12)} 
        + \| \grad w (t) \| _{L ^{\infty} (B _\frac12)}
    \right) \| \omega (t) \| _{L ^\frac32 (B _\frac12)} ^\frac12 \\
    &\le
    \left(
        2 + 
        \| w (t) \| _{L ^{\infty} (B _\frac12)} 
        + \| \grad w (t) \| _{L ^{\infty} (B _\frac12)}
    \right) 
    \\
    &\qquad 
    \times \left(
        \| \curl v (t) \| _{L ^\frac32 (B _\frac12)} ^\frac12
        + \| \curl w (t) \| _{L ^\frac32 (B _\frac12)} ^\frac12
    \right)
\end{align*}
since $u = w + v$, and $|v| \le 1$ inside $B _\frac12$. By \eqref{eqn:hypotheses-for-v-w},
\begin{align*}
    \int _{-\frac14} ^0 \Phi (t) \d t \lesssim \left(
        1 + \| w \| _{L ^\frac43 _t \Lip _x (Q _\frac12)}
    \right) \left(
        \| \grad v \| _{L ^2 (Q _\frac12)} ^\frac12 + \| \curl w (t) \| _{L ^2 _t L ^\frac32 _x (Q _\frac12)} ^\frac12
    \right) \le C.
\end{align*}
%
So by Gr\"onwall's inequality, 
\begin{align*}
    \| \omega ^\frac34 \| _{\LLLH (Q _\frac14)} ^2 \le e ^C - 1.
\end{align*}
\end{proof}

\begin{proof}[Proof of Proposition \ref{prop:omega-energy-space} (b)]
From Proposition \ref{prop:omega-energy-space} (a) and Sobolev embedding,
\begin{align*}
    \| \omega \| _{L ^\infty _t L ^\frac32 _x \cap L ^\frac32 _t L ^\frac92 _x (Q _\frac14)} \le C,
\end{align*}
this interpolates the space
\begin{align*}
    \| \omega \| _{L ^4 _t L ^2 _x (Q _\frac14)} \le C.
\end{align*}
Multiply the vorticity equation \eqref{eqn:vorticity} by $\vs ^2 \omega$ then integrate over $\R ^3$,
\begin{align*}
    \frac{\d}{\d t} \int \vs ^2 \frac{|\omega| ^2}{2} \d x + \int \vs ^2 |\grad \omega| ^2 \d x &= \int (\pt \vs ^2 + \La \vs ^2) \frac{|\omega| ^2}{2} \d x \\
    &\qquad - \int (u \cdot \grad \omega) \cdot \vs ^2 \omega \d x \\
    &\qquad + \int (\omega \cdot \grad u) \cdot \vs ^2 \omega \d x.
\end{align*}
The first integral is $L ^1$ in time because $\omega \in L ^4 _t L ^2 _x$. For the second, 
\begin{align*}
    \int (u \cdot \grad \omega) \cdot \vs ^2 \omega \d x &= \int \vs ^2 u \cdot \grad \frac{|\omega| ^2}{2} \d x \\
    &= -\int \frac{|\omega| ^2}{2} u \cdot \grad \vs ^2 \d x \\
    &= -\int \vs |\omega| ^2 u \cdot \grad \vs \\
    &\le \| \vs \omega \| _{L ^2} \|u \cdot \grad \vs |\omega| \| _{L ^2},
\end{align*}
the latter is bounded $L ^1$ in time, by $u \in L ^\frac43 _t L ^\infty _x$ and $\omega \in L ^4 _t L ^2 _x$. For the third integral,
\begin{align*}
    \int (\omega \cdot \grad u) \cdot \vs ^2 \omega \d x = \int (\omega \cdot \grad v) \cdot \vs ^2 \omega \d x + \int (\omega \cdot \grad w) \cdot \vs ^2 \omega \d x.
\end{align*}
$w$ is bounded in $L ^\frac43 _t \Lip _x$, and for $v$,
\begin{align*}
    \int (\omega \cdot \grad v) \cdot \vs ^2 \omega \d x &= \int v \cdot (\omega \cdot \grad) (\vs ^2 \omega) \d x \\
    &= \int v \cdot \omega (\omega \cdot \grad \vs ^2) \d x + \int v \cdot (\vs ^2 \omega \cdot \grad \omega) \d x.
\end{align*}
The former is $L ^1$ in time, while the latter can be bounded by Cauchy-Schwartz,
\begin{align*}
    \int v \cdot (\vs ^2 \omega \cdot \grad \omega) \d x 
    \le \frac12 \int |v \tensor \vs \omega| ^2 \d x 
    + \frac12 \int \vs ^2 |\grad \omega| ^2 \d x.
\end{align*}
In conclusion,
\begin{align*}
    &\frac{\d}{\d t} \int \vs ^2 \frac{|\omega| ^2}{2} \d x 
    + \frac12 \int \vs ^2 |\grad \omega| ^2 \d x
    \\
    &\qquad 
    \le C \| \omega (t) \| _{L ^2 (B _\frac14)} ^2 
    + C \| u (t) \| _{L ^\infty (B _\frac14)} \| \omega (t) \| _{L ^2 (B _\frac14)} \| \vs \omega (t) \| _{L ^2} \\
    &\qquad \qquad
    + C \| \grad w \| _{L ^\infty (B _\frac14)} \| \vs \omega (t) \| _{L ^2} ^2 \\
    &\qquad \le C\Phi (t) \left(
        1 + \int \vs ^2 \frac{|\omega| ^2}{2} \d x 
    \right)
\end{align*}
where
\begin{align*}
    \Phi (t) = \| \omega (t) \| _{L ^2 (B _\frac14)} ^2 
    + \| u (t) \| _{L ^\infty (B _\frac14)} \| \omega (t) \| _{L ^2 (B _\frac14)}
    + \| \grad w (t) \| _{L ^\infty (B _\frac14)},
\end{align*}
whose integral is bounded using \eqref{eqn:hypotheses-for-v-w},
\begin{align*}
    \int _{-\frac1{16}} ^0 \Phi (t) \d t \le \| \omega \| _{L ^2 (Q _\frac14)} ^2 
    + \| u \| _{L ^{\frac43} _t L ^\infty _x (Q _\frac14)} 
    \| \omega \| _{L ^4 _t L ^2 _x (Q _\frac14)}  
    + \| \grad w \| _{L ^{\frac43} _t L ^\infty _x (Q _\frac14)} \le C.
\end{align*}
By a Gr\"onwall argument, we have
\begin{align*}
    \| \omega \| _{\LLLH (Q _\frac18)} ^2 \le e ^C - 1.
\end{align*}
\end{proof}

\subsection{Bound Higher Derivatives in the Energy Space}

\newcommand{\vn}{v}
\newcommand{\wn}{w}
\newcommand{\vpn}{\vp _n}
\newcommand{\vpns}{\vps _n}

Now we iteratively show higher derivatives of vorticity by induction.



\newcommand{\BP}{\mathbf P}
\newcommand{\BPvk}{\BP _{v, k}}
\newcommand{\BPwk}{\BP _{w, k}}

\begin{proposition}
\label{prop:omega-energy-space-n}
For any $n \ge 1$, if $u = v + w$ in $Q _{8 ^{-n}}$, where $v, w$ are bounded in
\begin{align}
    \label{eqn:hypotheses-for-v-w-n}
    \| v \| _{L ^\infty (Q _{8 ^{-n} / 2})}
    + \| v \| _{L ^2 _t H ^{n + 1} _x (Q _{8 ^{-n} / 2})} \le c _n, \\
    \| w \| _{L ^\frac43 _t C ^{n + 1} _x (Q _{8 ^{-n} / 2})}
    \le c _n,
\end{align}
for some constant $c _n$, $\omega = \curl u$ solves the vorticity equation \eqref{eqn:vorticity}, and is bounded in
\begin{align}
    \label{eqn:induction-condition}
    \| \omega \| _{L ^\infty _t H ^{n - 1} _x \cap L ^2 _t H ^n _x (Q _{8 ^{-n} / 2})} \le c _n,
\end{align}
then for any multiindex $\alpha$ with $|\alpha| = n$,
\begin{enumerate}[\upshape (a)]
    \item $\| \grada \omega ^\frac34 \| _{\mathcal E (Q _{8 ^{-n} / 4})} \le C _n$
    \item $\| \grada \omega \| _{\mathcal E (Q _{8 ^{-n-1}})} \le C _n$
\end{enumerate}
for some $C _n$ depending on $c _n$ and $n$.
\end{proposition}

\begin{proof}[Proof of Proposition \ref{prop:omega-energy-space-n} (a)]
Similarly we fix smooth cut-off functions $\vrn$ and $\vsn$ which satisfy
\begin{align*}
    \ind{Q _{8 ^{-n-1}}} \le \vsn \le \ind{Q _{8 ^{-n} / 4}} \le \vrn \le \ind{Q _{8 ^{-n} / 2}}.
\end{align*}

Differentiate \eqref{eqn:vorticity} by $\grada$,
\begin{align}
    \label{eqn:omega-a}
    \pt \grada \omega + u \cdot \grad \grada \omega - \grada \omega \cdot \grad u + \BP _{\alpha} = \La \grada \omega,
\end{align}
where
\begin{align*}
    \BP _\alpha = \sum _{\beta < \alpha} \begin{pmatrix}
        \alpha \\ \beta
    \end{pmatrix} \curl \left(
        \gradb \omega \cross \grad ^{\alpha - \beta} u
    \right).
\end{align*}
Multiply \eqref{eqn:omega-a} by $\frac32\vrn ^6 (\grada \omega) ^\frac12$ then integrate in space, 
\begin{align*}
    &\frac{\d}{\d t} \int \vrn ^6 |\grada \omega| ^\frac32 \d x
    + \frac43 \int \vrn ^6 |\grad \grada \omega ^\frac34| ^2 \d x 
    \\
    &\qquad 
    \le \int \left[
        (\pt + u \cdot \grad + \La) \vrn ^6
    \right] |\grada \omega|^\frac32 \d x 
    \\
    &\qquad \qquad 
    + \int \vrn ^6 \grada \omega \cdot \grad w \cdot (\grada \omega) ^\frac12 \d x
    \\
    &\qquad \qquad 
    + \int \grada \omega \cdot ((\grada \omega) ^\frac12 \tensor \grad \vrn ^6) \cdot v \d x 
    \\
    &\qquad \qquad 
    + \|v\| _{L ^\infty (Q _{8 ^{-n}})} ^2
    \int \vrn ^6 |\grada \omega| ^\frac32 \d x
    + \int \vrn ^6 | \grad \grada \omega ^\frac34 | ^2 \d x 
    \\
    &\qquad \qquad
    + \frac32 \int \vrn ^6 (\grada \omega) ^\frac12 \cdot \BP _\alpha \d x
\end{align*}
same as in the proof of Proposition \ref{prop:omega-energy-space} (a). So
\begin{align*}
    &
    \frac{\d}{\d t} \int \vrn ^6 |\grada \omega| ^\frac32 \d x
    + \frac13 \int \vrn ^6 |\grad \grada \omega ^\frac34| ^2 \d x 
    \\
    &\qquad 
    \le C \left(
        1
        + \| u (t) \| _{L ^{\infty} (B _{8 ^{-n}})} 
        + \| \grad w (t) \| _{L ^{\infty} (B _{8 ^{-n}})} 
    \right) \int \vrn ^4 |\grada \omega| ^\frac32 \d x
    \\
    &\qquad \qquad
    + \frac32 \int \vrn ^6 (\grada \omega) ^\frac12 \cdot \BP _\alpha \d x.
\end{align*}
Terms other than $\BP _\alpha$ are dealt with by the same way as in Proposition \ref{prop:omega-energy-space}: 
\begin{align*}
    \int \vrn ^4 |\grada \omega| ^\frac32 \d x \le 
    \| \grada \omega (t) \| _{L ^\frac32 (B _{8^{-n}})} ^\frac12 \left(
        \int \vrn ^6 |\grada \omega| ^\frac32 \d x
    \right) ^\frac23.
\end{align*}
The induction condition \eqref{eqn:induction-condition} ensures that $\| \grada \omega \| _{L ^2 (Q _{8 ^{-n}})} \le c _n$. Therefore
\begin{align*}
    &\int _{-8 ^{-2n}} ^0 \left(
        1
        + \| u (t) \| _{L ^{\infty} (B _{8 ^{-n}})} 
        + \| \grad w (t) \| _{L ^{\infty} (B _{8 ^{-n}})} 
    \right) \| \grada \omega (t) \| _{L ^\frac32 (B _{8^{-n}})} ^\frac12
    \d t \\
    &\qquad \lesssim \left(
        1 + \| v \| _{L ^\infty (B _{8 ^{-n}})} + \| w \| _{L ^\frac43 _t C ^1 _x (B _{8 ^{-n}})}
    \right) \| \grada \omega \| _{L ^2 (Q _{8 ^{-n}})} ^{\frac12} 
    \le C _n.
\end{align*}

Now let's focus on $\BP _\alpha$. 
\begin{align*}
    | \BP _\alpha | &\lesssim \sum _{k = 0} ^n |\grad ^k \omega| |\grad ^{n - k + 1} u| 
    \le \sum _{k = 0} ^n |\grad ^k \omega| |\grad ^{n - k + 1} v| + \sum _{k = 0} ^n |\grad ^k \omega| |\grad ^{n - k + 1} w|. 
\end{align*}
We denote 
\begin{align*}
    \BPvk = |\grad ^k \omega| |\grad ^{n - k + 1} v|, \qquad 
    \BPwk = |\grad ^k \omega| |\grad ^{n - k + 1} w|.
\end{align*}
First we estimate $\BPvk$. By \eqref{eqn:hypotheses-for-v-w-n} and \eqref{eqn:induction-condition}, when $k = 0$,
\begin{align*}
    \|\BP _{v, 0}\| _{L ^1 _t L ^\frac32 _x (Q _{8 ^{-n}})} 
    \le \| \omega \| _{L ^2 _t L ^6 _x (Q _{8 ^{-n}})} 
    \| \grad ^{n + 1} v \| _{L ^2 _t L ^2 _x (Q _{8 ^{-n}})} \le C _n,
\end{align*}
and when $0 < k \le n$,
\begin{align*}
    \|\BPvk \| _{L ^1 _t L ^\frac32 _x (Q _{8 ^{-n}})} 
    \le \| \grad ^{k} \omega \| _{L ^2 _t L ^2 _x (Q _{8 ^{-n}})} 
    \| \grad ^{n + 1 - k} v \| _{L ^2 _t L ^6 _x (Q _{8 ^{-n}})} \le C _n.
\end{align*}
Next we estimate $\BPwk$. When $0 \le k < n$,
\begin{align*}
    \|\BPwk \| _{L ^1 _t L ^\frac32 _x (Q _{8 ^{-n}})} 
    \le \| \grad ^{k} \omega \| _{L ^\infty _t L ^2 _x (Q _{8 ^{-n}})} 
    \| \grad ^{n + 1 - k} w \| _{L ^\frac43 _t L ^\infty _x (Q _{8 ^{-n}})} \le C _n.
\end{align*}
Finally, when $k = n$,
\begin{align*}
    | \BP _{w, n} | _{L ^\frac32 _x (B _{8 ^{-n}})}
    \le |\grad ^n \omega | | \grad w |.
\end{align*}
Therefore,
\begin{align*}
    &\int \vrn ^6 (\grada \omega) ^\frac12 \cdot \BP _\alpha \d x \le \left(
        1 + \int \vrn ^6 |\grada \omega| ^\frac32 \d x
    \right) \\
    &\qquad \times \left(
        \sum _{k = 0} ^n \|\BPvk \| _{L ^\frac32 _x (B _{8 ^{-n}})} 
        + \sum _{k = 0} ^{n - 1} \|\BPwk \| _{L ^\frac32 _x (B _{8 ^{-n}})} 
        + \| \grad w \| _{L ^\infty _x (B _{8 ^{-n}})}
    \right)
\end{align*}
In conclusion, we have shown that
\begin{align*}
    &\frac{\d}{\d t} \int \vrn ^6 |\grada \omega| ^\frac32 \d x
    + \frac13 \int \vrn ^6 |\grad \grada \omega ^\frac34| ^2 \d x 
    \le C \Phi (t) \left(
        1 + \int \vrn ^6 |\grad ^n \omega| ^\frac32 \d x
    \right),
\end{align*}
where
\begin{align*}
    \Phi (t) &= \left(
        1
        + \| u (t) \| _{L ^{\infty} (B _{8 ^{-n}})} 
        + \| \grad w (t) \| _{L ^{\infty} (B _{8 ^{-n}})} 
    \right) \| \grada \omega (t) \| _{L ^\frac32 (B _{8^{-n}})} ^\frac12 \\
    &\qquad + \sum _{k = 0} ^n \|\BPvk \| _{L ^\frac32 _x (B _{8 ^{-n}})} 
        + \sum _{k = 0} ^{n - 1} \|\BPwk \| _{L ^\frac32 _x (B _{8 ^{-n}})} 
        + \| \grad w \| _{L ^\infty _x (B _{8 ^{-n}})}
\end{align*}
with integral
\begin{align*}
    \int _{-8 ^{-2n} / 4} ^0 \Phi (t) \d t \le C _n.
\end{align*}

Taking the sum over all multi-index $\alpha$ with size $|\alpha| = n$, we have
\begin{align*}
    &\frac{\d}{\d t} \int \vrn ^6 |\grad ^{n} \omega| ^\frac32 \d x
    + \frac13 \int \vrn ^6 |\grad ^{n + 1} \omega ^\frac34| ^2 \d x 
    \le C \Phi (t) \left(
        1 + \int \vrn ^6 |\grad ^{n + 1} \omega| ^\frac32 \d x
    \right),
\end{align*}
Finally, Gr\"onwall inequality gives 
\begin{align*}
    \||\grad ^{n + 1} \omega| ^\frac34\| _{\LLLH (Q _{8 ^{-n} / 4})} \le C _n.
\end{align*}
\end{proof}

\begin{proof}[Proof of Proposition \ref{prop:omega-energy-space-n} (b)]
Now we multiply \eqref{eqn:omega-a} by $\vsn ^2 \grada \omega$ then integrate over $\R ^3$,
\begin{align*}
    \frac{\d}{\d t} \int \vsn ^2 \frac{|\grada \omega| ^2}{2} \d x + \int \vsn ^2 |\grad \grada \omega| ^2 \d x &= \int (\pt \vsn ^2 + \La \vsn ^2) \frac{|\grada \omega| ^2}{2} \d x \\
    &\qquad - \int (u \cdot \grad \grada \omega) \cdot \vsn ^2 \grada \omega \d x \\
    &\qquad + \int (\grada \omega \cdot \grad u) \cdot \vsn ^2 \grada \omega \d x \\
    &\qquad + \int \vsn ^2 \grada \omega \cdot \BP _\alpha \d x
\end{align*}
For the same reason, the only term that we need to take care of is $\BP _\alpha$ term, and the others are dealt the same as in Proposition \ref{prop:omega-energy-space} (b): 
\begin{align*}
    &\int (\pt \vsn ^2 + \La \vsn ^2) \frac{|\grada \omega| ^2}{2} \d x
     - \int (u \cdot \grad \grada \omega) \cdot \vsn ^2 \grada \omega \d x + \int (\grada \omega \cdot \grad u) \cdot \vsn ^2 \grada \omega \d x \\
    & \lesssim _n
    \| \grada \omega \| _{L ^2 (Q _{8 ^{-n} / 4})} ^2 + 
    \| u \| _{L ^\infty (Q _{8 ^{-n} / 4})} 
    \| \grada \omega \| _{L ^2 (Q _{8 ^{-n} / 4})} 
    \left( 
        \int \vsn ^2 \frac{| \grada \omega | ^2}2 \d x 
    \right) ^\frac12 \\
    &\qquad 
    + \| \grad w  \| _{L ^\infty (Q _{8 ^{-n} / 4})}
    \int \vsn ^2 \frac{| \grada \omega | ^2}2 \d x 
    + \| v \| _{L ^\infty (Q _{8 ^{-n} / 4})} 
    \| \grada \omega \| _{L ^2 (Q _{8 ^{-n} / 4})} ^2 \\
    &\qquad 
    + \frac1\e \| v \| _{L ^\infty (Q _{8 ^{-n} / 4})} ^2 
    \int \vsn ^2 \frac{| \grada \omega | ^2}2 \d x
    + \e \int \vsn ^2 |\grad \grada \omega| ^2 \d x.
\end{align*}
The last term can be absorbed into the left, and we will use Gr\"onwall on the remaining terms. 

Now we shall focus on the $\BP _\alpha$ term. From Proposition \ref{prop:omega-energy-space-n} (a), we have
\begin{align}
    \label{eqn:gradn-omega}
    \| \grad ^n \omega \| _{L ^\infty _t L ^\frac32 _x \cap L ^\frac32 _x L ^\frac92 _t (Q _{8 ^{-n} / 4})} \le C _n.
\end{align}
Again by interpolation,
\begin{align*}
    \| \grad ^n \omega \| _{L ^4 _t L ^2 _x (Q _{8 ^{-n} / 4})} \le C _n, \qquad \| \grad ^n \omega \| _{L ^2 _t L ^3 _x (Q _{8 ^{-n} / 4})} \le C _n, 
\end{align*}
First we estimate $\BPwk$. In this case, for any $0 \le k \le n$,
\begin{align*}
    \| \BPwk \| _{L ^1 _t L ^2 _x (Q _{8 ^{-n} / 4})} 
    \le \| \grad ^{k} \omega \| _{L ^4 _t L ^2 _x (Q _{8 ^{-n} / 4})} 
    \| \grad ^{n + 1 - k} w \| _{L ^\frac43 _t L ^\infty _x (Q _{8 ^{-n}})} \le C _n.
\end{align*}
Then we estimate $\BPvk$. When $0 < k \le n$,
\begin{align*}
    \|\BPvk \| _{L ^1 _t L ^2 _x (Q _{8 ^{-n}})} 
    \le \| \grad ^{k} \omega \| _{L ^2 _t L ^3 _x (Q _{8 ^{-n}})} 
    \| \grad ^{n + 1 - k} v \| _{L ^2 _t L ^6 _x (Q _{8 ^{-n}})} \le C _n.
\end{align*}
For the case $k = 0$ of the $\vn$ term, we put the curl on $\grada \omega$,
\begin{align*}
    &\int \vsn ^2 \grada \omega \cdot \curl \left(
        \omega \cross \grada \vn
    \right) \d x
    \\
    &\qquad 
    = \-\int \left(
        \omega \cross \grada \vn
    \right) \cdot \curl (\vsn ^2 \grada \omega) \d x 
    \\
    &\qquad
    \le \int 
    \vsn ^2 |\omega| 
    |\grada \vn| 
    |\grad \grada \omega| 
    + \vsn |\grad \vsn| |\omega| 
    |\grada \vn| 
    |\grada \omega| 
    \d x
    \\
    &\qquad
    \le \int \vsn ^2 |\omega| ^2 |\grada \vn| ^2 \d x
    + \e \int \vsn ^2 |\grad \grada \omega| ^2 \d x 
    + \frac1\e \int 
    |\grad \vsn| ^2 |\grada \omega| ^2
    \d x.
\end{align*}
where $|\grad \grada \omega|$ term can be absorbed to the left. By \eqref{eqn:gradn-omega} and Sobolev embedding,
\begin{align*}
    \| \omega \| _{L ^\infty _t L ^3 _x  (Q _{8 ^{-n} / 4})} \le C _n.
\end{align*}
Therefore
\begin{align*}
    \iint \vsn ^2 |\omega| ^2 |\grada \vn| ^2 \d x \d t
    \le \| \omega \|  _{L ^\infty _t L ^3 _x  (Q _{8 ^{-n} / 4})} ^2 
    \| \grada \vn \| _{L ^2 _t L ^6 _x  (Q _{8 ^{-n} / 4})} ^2 \le C _n.
\end{align*}
In conclusion,
\begin{align*}
    \frac{\d}{\d t} \int \vsn ^2 \frac{|\grada \omega| ^2}{2} \d x + \int \vsn ^2 |\grad \grada \omega| ^2 \d x \le C \Phi (t) \left(
        1 + \int \vsn ^2 \frac{|\grada \omega| ^2}{2} \d x
    \right),
\end{align*}
where
\begin{align*}
    \Phi (t) &= \| \grada \omega (t) \| _{L ^2 (B _{8 ^{-n} / 4})} ^2 + 
    \| u \| _{L ^\infty (B _{8 ^{-n} / 4})} 
    \| \grada \omega \| _{L ^2 (B _{8 ^{-n} / 4})} 
     \\
    &\qquad 
    + \| \grad w  \| _{L ^\infty (B _{8 ^{-n} / 4})}
    + \| v \| _{L ^\infty (B _{8 ^{-n} / 4})} 
    \| \grada \omega \| _{L ^2 (B _{8 ^{-n} / 4})} ^2 \\
    &\qquad 
    + \frac1\e \| v \| _{L ^\infty (B _{8 ^{-n} / 4})} ^2 
    + \sum _{k = 0} ^n \| \BPwk \| _{L ^2 (B _{8 ^{-n} / 4})} 
    + \sum _{k = 0} ^{n - 1} \| \BPvk \| _{L ^2 (B _{8 ^{-n} / 4})} \\
    &\qquad 
    + \| \omega \|  _{L ^3  (B _{8 ^{-n} / 4})} ^2 
    \| \grada \vn \| _{L ^6  (B _{8 ^{-n} / 4})} ^2 
    + \frac1\e \| \grada \omega (t) \| _{L ^2 (B _{8 ^{-n} / 4})} ^2
\end{align*}
has integral $\int _{-8 ^{-2n} / 16} ^0 \Phi (t) \d t \le C _n$.
Finally Gr\"onwall inequality gives 
\begin{align*}
    \|\grada \omega\| _{\LLLH (Q _{8 ^{-n-1}})} \le C _{n + 1}.
\end{align*}
\end{proof}

\subsection{Proof of the Local Theorem}

\begin{proof}[Proof of the Local Theorem \ref{thm:local}]
First, Proposition \ref{prop:v-L-inf-L-2} gives
\begin{align*}
    \| v \| _{\mathcal E (Q _1)} \le \eta
\end{align*}
where $\eta$ can be chosen arbitrarily small if we pick $\eta _1$ small. Next, by Proposition \ref{prop:v-L-inf}, we know
\begin{align*}
    \| v \| _{L ^\infty (Q _\frac12)} \le 1.
\end{align*}
These two steps implies \eqref{eqn:hypotheses-for-v-w}. As for \eqref{eqn:hypotheses-for-v-w-2}, $\curl w = \varpi$ in $B _1$, so we use interpolation in \eqref{eqn:integrability-varpi}: 
\begin{align*}
    \| \curl w \| _{L ^2 _t L ^\frac32 _x (Q _\frac12)} \le 
    \| \varpi \| _{L ^2 _t L ^{\frac{12}7} _x} \le \| \varpi \| _{\Lp1 \Lq3} ^\frac12
    \| \varpi \| _{\Lp2 \Lq4} ^\frac12 \le \eta
\end{align*}
$w$ is harmonic inside $B _1$, therefore
\begin{align*}
    \| w \| _{L ^\frac43 _t C ^n _x (Q _\frac12)} \lesssim _n \| w \| _{L ^\frac43 _t L ^1 _x (Q _1)} \le \eta
\end{align*}
due to \eqref{eqn:integrability-w} and $p _1 \ge \frac43$. Therefore, we can use Proposition \ref{prop:omega-energy-space} to obtain 
\begin{align*}
    \| \omega \| _{\mathcal E (Q _\frac18)} \le C.
\end{align*}

The next step is to use Proposition \ref{prop:omega-energy-space-n} iteratively. Suppose for $n \ge 1$ we know that 
\begin{align*}
    \| \grad ^{n - 1} \omega \| _{\mathcal E (Q _{8 ^{-n}})} \le c _{n}
\end{align*}
which is equivalent to \eqref{eqn:induction-condition}. Let $\vpn$ and $\vpns$ be a pair of smooth spatial cut-off functions, with 
\begin{align*}
    \ind{B _\frac{1}{8 ^n + 4}} \le \vpn \le \ind{B _\frac{1}{8 ^n + 3}}, \qquad 
    \ind{B _\frac{1}{8 ^n + 2}} \le \vpns \le \ind{B _\frac{1}{8 ^n + 1}},
\end{align*}
and set 
\begin{align*}
    v _n := -\curl \vpns \La \inv \vpn \omega, \qquad w _n = \vpn u - v _n.
\end{align*}
On the one hand, $\grad v _n$ is a Riesz transform of $\vp _n \omega$ up to lower order terms, so by the boundedness of Riesz transform we know 
\begin{align*}
    \| \grad ^{n + 1} v _n \| _{L ^2 (Q _{8 ^{-n} / 2})} \le \| \grad ^{n} \omega \| _{L ^2 (Q _{8 ^{-n}})} \le c _{n - 1}.
\end{align*}
On the other hand, we have similar boundedness estimates following Proposition \ref{prop:v-L-inf} as before,
\begin{align*}
    \| v _n \| _{L ^\infty (Q _{8 ^{-n} / 2})} \le 1.
\end{align*}
$w _n$ is harmonic in $B _{\frac1{8^n + 4}}$, so we also have
\begin{align*}
    \| w _n \| _{L ^\frac43 _t C ^{n + 1} _x (Q _{8 ^{-n} / 2})} \lesssim _n \| w _n \| _{L ^\frac43 _t L ^1 _x (Q _{\frac1{8^n + 4}})} \le \eta.
\end{align*}
Therefore, by Proposition \ref{prop:omega-energy-space-n}
\begin{align*}
    \| \grad ^n \omega \| _{\mathcal E (Q _{8 ^{-n - 1}})} \le C _n.
\end{align*}
By induction, we have 
\begin{align*}
    \| \grad ^n \omega \| _{\LLLH (Q _{8 ^{-n-1}})} \le C _n 
\end{align*}
for any $n$. By Sobolev embedding, this implies for any $n$,
\begin{align*}
    \| \grad ^n \omega \| _{L ^\infty (Q _{8 ^{-n-3}})} \lesssim 
    \| \grad ^n \omega \| _{L ^\infty _t L ^2 _x (Q _{8 ^{-n-3}})} +
    \| \grad ^{n + 2} \omega \| _{L ^\infty _t L ^2 _x (Q _{8 ^{-n-3}})} \le C _n .
\end{align*}
\end{proof}

\begin{appendix}
\section{Suitability of Solutions}

\begin{theorem}
Let $u$ be a suitable weak solution to the Navier-Stokes equation in $\R ^3$. That is, $u \in \LLLH$ solves the following equation
\begin{align}
    \label{eqn:u-original}
    \pt u + u \cdot \grad u + \grad P &= \La u, \qquad 
    \div u = 0
\end{align}
where $P$ is the pressure, and $u$ satisfies the following local energy inequality,
\begin{align}
    \label{eqn:u-original-suitability}
    \pt \uuhalf + \div \left(
        u \left(
            \uuhalf + P
        \right)
    \right) + |\grad u| ^2
    \le \La \uuhalf.
\end{align} 
Suppose $v \in \LLLH$ is compactly supported in space and solves the following equation,
\begin{align}
    \label{eqn:v-original}
    \pt v + \omega \cross v + \grad \RR (u \tensor v) &= \La v + \Cv, \qquad
    \div v = 0
\end{align}
where $\omega = \curl u$ is the vorticity, $\Cv \in L ^2 _t L ^\frac32 _{\loc,x}$ is a force term, and
\begin{align*}
    \RR = \frac12 \tr -\La \inv \div \div    
\end{align*}
is a symmetric Riesz operator. Moreover, suppose $v$ differs from $\vp u$ by 
\begin{align*}
    \vp u - v = w \in L ^\infty _t H ^1 _x \cap L ^2 _t H ^2 _x
\end{align*}
for some fixed $\vp \in C _c ^\infty (\R ^3)$.
Then $v$ satisfies the following local energy inequality,
\begin{align}
    \label{eqn:v-suitability}
    \pt \vvhalf + \div \left(v
        \RR (u \tensor v)
    \right) + |\grad v| ^2
    &\le \La \vvhalf + v \cdot \Cv.
\end{align}
\end{theorem}

\begin{proof}
It is well-known that the pressure $P$ can be recovered from $u$ by
\begin{align*}
    P = -\La \inv \div \div (u \tensor u).
\end{align*}
Since
\begin{align*}
    u \cdot \grad u + \grad P 
    &= \grad \uuhalf + \omega \cross u - \grad \La \inv \div \div (u \tensor u) \\
    &= \omega \cross u + \grad \RR (u \tensor u),
\end{align*}
The Navier-Stokes equation \eqref{eqn:u-original} can be rewritten as
\begin{align}
    \label{eqn:u}
    \pt u + \omega \cross u + \grad \RR (u \tensor u) &= \La u, 
\end{align}
and local energy inequality \eqref{eqn:u-original-suitability} can be rewritten as
\begin{align}
    \label{eqn:u-suitability}
    \pt \uuhalf + \div \left(
        u \RR (u \tensor u)
    \right) + |\grad u| ^2
    \le \La \uuhalf,
\end{align}

First, multiply \eqref{eqn:u} by $\vp$, 
\begin{align*}
    \pt \vp u + \omega \cross \vp u + \grad \RR (u \tensor \vp u) = \La (\vp u) + [\grad \RR, \vp] (u \tensor u) + [\vp, \La] u.
\end{align*}
Denote
\begin{align*}
    \Cu = [\grad \RR, \vp] (u \tensor u) + [\vp, \La] u
\end{align*}
for these commutator terms. Subtracting the equation of $v$ from this equation of $\vp u$, we will have the equation for $w$. In summary,
\begin{align}
    \label{eqn:vpu}
    \pt \vp u + \omega \cross \vp u + \grad \RR (u \tensor \vp u) &= \La (\vp u) + \Cu, \\
    \label{eqn:v} 
    \pt v + \omega \cross v + \grad \RR (u \tensor v) &= \La v + \Cv, \\
    \label{eqn:w}
    \pt w + \omega \cross w + \grad \RR (u \tensor w) &= \La w + \Cu - \Cv.
\end{align}
Recall from \cite{Vasseur2010} that $\La u \in L ^{\frac43-\e} _{\loc (t, x)}$. Since $\La w \in L ^2 _{t, x}$, we have $\La v \in L ^{\frac43-\e} _{\loc (t, x)}$. Therefore, we can multiply \eqref{eqn:vpu} and \eqref{eqn:v} by $w$, and \eqref{eqn:w} by $\vp u$ and $v$,
\begin{align}
    \label{eqn:sum-1}
    w \cdot \pt (\vp u) 
    + w \cdot \omega \cross \vp u 
    + w \cdot \grad \RR (u \tensor \vp u) 
    &= 
    w \cdot \La (\vp u) 
    + w \cdot \Cu,
    \\
    w \cdot \pt v 
    + w \cdot \omega \cross v 
    + w \cdot \grad \RR (u \tensor v) 
    &= 
    w \cdot \La v 
    + w \cdot \Cv
    \\
    \vp u \cdot \pt w 
    + \vp u \cdot \omega \cross w 
    + \vp u \cdot \grad \RR (u \tensor w) 
    &=
    \vp u \cdot \La w 
    + \vp u \cdot (\Cu - \Cv).
    \\
    \label{eqn:sum-4}
    v \cdot \pt w 
    + v \cdot \omega \cross w 
    + v \cdot \grad \RR (u \tensor w) 
    &=
    v \cdot \La w 
    + v \cdot (\Cu - \Cv).
\end{align}
Now take the sum of \eqref{eqn:sum-1}-\eqref{eqn:sum-4}. $\pt$ terms are
\begin{align*}
    &\vp u \cdot \pt w + w \cdot \pt (\vp u) + v \cdot \pt w + w \cdot \pt v \\
    &\qquad = \pt (\vp u \cdot w) + \pt (w \cdot v) \\
    &\qquad = \pt (|\vp u| ^2 - |v| ^2).
\end{align*}
$\omega \cross$ terms are
\begin{align*}
    w \cdot \omega \cross \vp u 
    + \vp u \cdot \omega \cross w 
    + w \cdot \omega \cross v 
    + v \cdot \omega \cross w 
    = 0.
\end{align*}
$\grad \RR$ terms are
\begin{align*}
    &
    w \cdot \grad \RR (u \tensor \vp u)
    + v \cdot \grad \RR (u \tensor w)
    \\
    &\qquad 
    + \vp u \cdot \grad \RR (u \tensor w)
    + w \cdot \grad \RR (u \tensor v) 
    \\
    &= 
    \div (w \RR (u \tensor \vp u))
    + \div (v \RR (u \tensor w)) 
    \\
    &\qquad
    + \div (\vp u \RR(u \tensor w))
    + \div (w \RR (u \tensor v)) \\
    &\qquad 
    - \div (w) \grad \RR (u \tensor \vp u)
    - \div (v) \grad \RR (u \tensor w)
    \\
    &\qquad 
    - \div (\vp u) \grad \RR (u \tensor w)
    - \div (\vp ) \grad \RR (u \tensor v) 
    \\
    &= 2 \div (\vp u \RR (u \tensor \vp u)
    - v \RR (u \tensor v))
    \\
    &\qquad 
    - (u \cdot \grad \vp) \left(
        \grad \RR (u \tensor \vp u)
        + \grad \RR (u \tensor w)
        + \grad \RR (u \tensor v) 
    \right)
    \\
    &= 2 \div (\vp u \RR (u \tensor \vp u)
    - v \RR (u \tensor v)) - 2 (u \cdot \grad \vp) \RR (u \tensor \vp u).
\end{align*}
Here we use $\div v = 0, \div (\vp u) = \div w = u \cdot \grad \vp$. $\La$ terms are
\begin{align*}
    &\vp u \cdot \La w + w \cdot \La (\vp u) + v \cdot \La w + w \cdot \La v \\
    &\qquad = \La (u \cdot w) - 2 \grad (\vp u) : \grad w + \La (v \cdot w) - 2 \grad v : \grad w \\
    &\qquad = \La (|\vp u| ^2 - |v| ^2) - 2 (|\grad (\vp u)| ^2 - |\grad v| ^2).
\end{align*}
Commutator terms are
\begin{align*}
    w \cdot \Cu
    + \vp u \cdot (\Cu - \Cv)
    + w \cdot \Cv
    + v \cdot (\Cu - \Cv)
    = 2 \vp u \cdot \Cu - 2 v \cdot \Cv.
\end{align*}
In summary, half the sum of these four identities \eqref{eqn:sum-1}-\eqref{eqn:sum-4} gives
\begin{align}
    \label{eqn:difference}
    &\pt \frac{|\vp u| ^2 - |v| ^2}{2} 
    + \div (\vp u \RR (u \tensor \vp u)
    - v \RR (u \tensor v))
    + |\grad (\vp u)| ^2 - |\grad v| ^2
    \\
    \notag
    &\qquad =
    \La \frac{|\vp u| ^2 - |v| ^2}{2} 
    + \vp u \cdot \Cu - v \cdot \Cv + (u \cdot \grad \vp) \RR (\vp u \tensor u).
\end{align}

Next, multiply local energy inequality of $u$ \eqref{eqn:u-suitability} by $\vp ^2$,
\begin{align}
    \notag
    &
    \pt \frac{|\vp u| ^2}{2} 
    + |\vp \grad u| ^2 
    + \div \left(
        \vp ^2 u \RR (u \tensor u)
    \right) 
    \\
    \notag
    &\qquad 
    \le \La \frac{|\vp u| ^2}{2} 
    + [\vp ^2, \La] \uuhalf 
    + [\div, \vp ^2] \left(
        u \RR (u \tensor u)
    \right), 
    \\
    \notag
    &
    \pt \frac{|\vp u| ^2}{2} 
    + |\grad (\vp u)| ^2 
    + \div \left(
        \vp u \RR (u \tensor \vp u)
    \right) \\
    \label{eqn:vpu-suitability}
    &\qquad 
    \le \La \frac{|\vp u| ^2}{2} 
    + [\vp ^2, \La] \uuhalf 
    + |u \tensor \grad \vp| ^2 
    + 2 (u \tensor \grad \vp) : (\vp \grad u)
    \\
    \notag
    &\qquad \qquad
    + [\div, \vp ^2] \left(
        u \RR (u \tensor u)
    \right) + \div (\vp u [\RR, \vp] (u \tensor u)).
\end{align}
The quadratic commutator terms in \eqref{eqn:vpu-suitability} are
\begin{align*} 
    & [\vp ^2, \La] \uuhalf 
    + |u \tensor \grad \vp| ^2 
    + 2 (u \tensor \grad \vp) : (\vp \grad u) \\
    &\qquad
    = [\vp ^2, \La] \uuhalf 
    + |u| ^2 |\grad \vp| ^2 
    + 2 \grad \vp \cdot \vp \grad u \cdot u \\
    &\qquad
    = -2 \grad(\vp ^2) \cdot \grad \uuhalf 
    - \La (\vp ^2) \uuhalf
    + |u| ^2 |\grad \vp| ^2 
    + 2 \grad \vp \cdot \grad u \cdot \vp u \\
    &\qquad= 
    -4 \vp \grad \vp \cdot \grad \uuhalf 
    - \frac12 \La (\vp ^2) |u| ^2 
    + |u| ^2 |\grad \vp| ^2 
    + 2 \grad \vp \cdot \grad u \cdot \vp u \\
    &\qquad= -2 \vp \grad \vp \cdot \grad u \cdot u- \vp \La \vp |u| ^2 \\
    &\qquad= \vp u \cdot (-2 \grad \vp \cdot \grad u - (\La \vp) u) \\
    &\qquad= \vp u \cdot [\vp, \La] u.
\end{align*}
The cubic commutator terms in \eqref{eqn:vpu-suitability} are
\begin{align*}
    &
    [\div, \vp ^2] \left(
        u \RR (u \tensor u)
    \right) 
    + \div (\vp u [\RR, \vp] (u \tensor u)) 
    \\
    &\qquad
    = 2 \vp \grad \vp \cdot u \RR (u \tensor u) 
    + \vp u \cdot \grad [\RR, \vp] (u \tensor u) 
    + \div (\vp u) [\RR, \vp] (u \tensor u)
    \\
    &\qquad
    = 2 \vp (u \cdot \grad \vp)  \RR (u \tensor u) 
    + \vp u \cdot \grad [\RR, \vp] (u \tensor u) 
    + (u \cdot \grad \vp) [\RR, \vp] (u \tensor u)
    \\
    &\qquad
    = 2 \vp (u \cdot \grad \vp)  \RR (u \tensor u) 
    + \vp u \cdot \grad [\RR, \vp] (u \tensor u) 
    \\
    &\qquad \qquad
    + (u \cdot \grad \vp) \RR (\vp u \tensor u)
    - (u \cdot \grad \vp) \vp \RR ( u \tensor u)
    \\
    &\qquad
    = \vp u \cdot \grad \vp \RR (u \tensor u) 
    + \vp u \cdot \grad [\RR, \vp] (u \tensor u) 
    + (u \cdot \grad \vp) \RR (\vp u \tensor u)
    \\
    &\qquad
    = \vp u \cdot [\grad, \vp] \RR (u \tensor u) 
    + \vp u \cdot \grad [\RR, \vp] (u \tensor u) 
    + (u \cdot \grad \vp) \RR (\vp u \tensor u)
    \\
    &\qquad
    = \vp u \cdot \left(
        [\grad, \vp] \RR - \grad [\vp, \RR]
    \right)
    (u \tensor u) 
    + (u \cdot \grad \vp) \RR (\vp u \tensor u)
    \\
    &\qquad
    = \vp u \cdot 
        [\grad \RR, \vp]
    (u \tensor u) 
    + (u \cdot \grad \vp) \RR (\vp u \tensor u).
\end{align*}
Therefore, local energy inequality for $\vp u$ can be simplified as
\begin{align*}
    &\pt \frac{|\vp u| ^2}{2} 
    + |\grad (\vp u)| ^2 
    + \div \left(
        \vp u \RR (u \tensor \vp u)
    \right) \\
    &\qquad
    \le \La \frac{|\vp u| ^2}{2} 
    + \vp u \cdot \Cu + (u \cdot \grad \vp) \RR (\vp u \tensor u).
\end{align*}
Subtracting \eqref{eqn:difference} from this, we obtain \eqref{eqn:v-suitability}.
\end{proof}

\end{appendix}

\bibliographystyle{alpha}
\bibliography{mathscinet}

\end{document}